\documentclass[reqno,11pt]{amsart}
\numberwithin{equation}{section}

\usepackage{verbatim}
\usepackage[utf8]{inputenc}
\usepackage[italian,english]{babel}
\usepackage{esint}
\usepackage{color}
\usepackage{mathrsfs}
\usepackage{cancel}
\usepackage{mathtools}
\usepackage{bigints}
\usepackage{calrsfs}
\usepackage{array}
\usepackage{graphicx}
\usepackage{float}
\mathtoolsset{showonlyrefs}
\numberwithin{equation}{section}
\newcommand{\R}{\mathbb{R}}

\newcommand{\Sf}{\mathbb{S}}

\newenvironment{system}
  {\left\lbrace\begin{array}{*{4}{c@{}>{{}}c<{{}}@{}}c}}
  {\end{array}\right.}
\newcommand{\restr}[1]{\lower3pt\hbox{$|_{#1}$}}

\newcommand{\dive}{{\mathrm{div}}}
\newcommand{\Haus}[1]{{\mathscr H}^{#1}}     % Misura di Hausdorff
\newcommand{\la}{{\langle}}                  % brackets
\newcommand{\ra}{{\rangle}}
\newcommand{\nchi}{{\raise.3ex\hbox{$\chi$}}}

\newtheorem{theorem}{Theorem}[section]

\newtheorem{corollary}[theorem]{Corollary}
\newtheorem{lemma}[theorem]{Lemma}
\newtheorem{proposition}[theorem]{Proposition}

\newtheorem{notation}[theorem]{Notation}

\newtheorem{remark}[theorem]{Remark}

\DeclarePairedDelimiter\abs{\lvert}{\rvert}
\DeclarePairedDelimiter\bigabs{\Big\lvert}{\Big\rvert}
\DeclarePairedDelimiter\norm{\lVert}{\rVert}
\DeclarePairedDelimiter\bignorm{\Big\lVert}{\Big\rVert}

\newcommand{\ric}{\mathop {\rm Ric}\nolimits}

\begin{document}
\title{Geometric aspects of p-capacitary potentials}
\author[Mattia Fogagnolo]{Mattia Fogagnolo}
\address[Mattia Fogagnolo]{Università degli Studi di Trento, Via Sommarive 14, 38123 Povo (TN), Italy}
\email{mattia.fogagnolo@unitn.it}
\author[Lorenzo Mazzieri]{Lorenzo Mazzieri}
\address[Lorenzo Mazzieri]{Università degli Studi di Trento, Via Sommarive 14, 38123 Povo (TN), Italy}
\email{lorenzo.mazzieri@unitn.it}
\author[Andrea Pinamonti]{Andrea Pinamonti}
\address[Andrea Pinamonti]{Università degli Studi di Trento, Via Sommarive 14, 38123 Povo (TN), Italy}
\email{andrea.pinamonti@unitn.it}

%\maketitle
\begin{abstract} 
We provide monotonicity formulas for solutions to the $p$-Laplace equation defined in the exterior of a convex domain. A number of analytic and geometric consequences are derived, including the classical Minkowski inequality as well as new characterizations of rotationally symmetric solutions and domains. The proofs rely on the conformal splitting technique introduced by the second author in collaboration with V. Agostiniani. 
\end{abstract}

%\tableofcontents
\maketitle
\section{Introduction}
Given a convex bounded domain $\Omega \subset \R^n$, $n \geq 3$, with smooth boundary and $1<p<n$, we consider the associated $p$-\emph{capacitary potentials}, namely the unique solution $u$ to the following problem
\begin{align}\label{prob_ex}
\left\{\begin{array}{lll}
\Delta_p u=0 & \mbox{in} & \R^n\setminus \overline{\Omega}\\
\quad \,\,u=1 & \mbox{on} & \partial\Omega\\
\, u(x)\to 0 & \mbox{as} & |x|\to\infty,
\end{array}\right.
\end{align}
where $\Delta_p u$ is the $p$-Laplace operator, that is
\[
\Delta_p u=\dive(|D u|^{p-2}D u).
\] 
A classical result by Lewis \cite{Lew} guarantees that the solution is smooth and $|Du|\neq 0$ in $\R^n\setminus \overline{\Omega}$. We consider the following functions
\begin{equation}
\label{defvqp}
V_q^p(t)=
\begin{system}
   &\left(\dfrac{{C_p(\Omega)}}{t^{p-1}}\right)^{\frac{(n-1)(q-1)}{(n-p)}}\!\!\!\!\!\!\!\!\!\int_{\{u=t\}}\abs{Du}^{q(p-1)} \,d\sigma& \qquad \qquad &\mbox{if}&\qquad 0\leq q <\infty\\
	 & \sup\limits_{{\{u=t}\}} \dfrac{\abs{Du}}{u^{\frac{n-1}{n-p}}} & \qquad &\mbox{if}&\qquad q=\infty,
\end{system}
\end{equation}
where $\sigma = \mathcal{H}^{n-1}$ denotes the $(n-1)-$dimensional Hausdorff measure, $q\in [1,\infty)$ and $C_p(\Omega)$ is the (re-scaled) $p$-capacity of $\Omega$, defined as 
\begin{equation}
\label{capacity}
C_p(\Omega)=\inf\left\{\frac{1}{\Big(\frac{n-p}{p-1}\Big)^{p-1}\abs{\Sf^{n-1}}}\int_{\mathbb{R}^n} \!|D v|^p d\mu \ \bigg\vert \ v\in C^{\infty}_c(\mathbb{R}^n),\ v\geq 1\ \mbox{on}\ \Omega\right\}.
\end{equation}
%where by $\Sf^{n-1}$ we denote the $n-1$ dimensional sphere.
%Our main results show that for a suitable range of parameters $(p, q)$ the above defined quantities  are non-decreasing. Moreover, we prove that if monotonicity is not strict then $\Omega$ is forced to be a ball and $u$ to be rotationally symmetric. 
We denote by $\mu$ the Lebesgue measure $\mathcal{L}^n$.
Notice that, if $\Omega$ is a ball of radius $R$, the only solution to problem \eqref{prob_ex} is given by
\begin{equation}
\label{urot}
u(x)=\bigg(\frac{R}{\abs{x}}\bigg)^{\!\!\frac{n-p}{p-1}},
\end{equation}
and straightforward computations show that $V_q^p$ and $V_\infty^p$ are actually constant in this case.
In fact, our main results show that for any choice of the parameters $(p, q)$ in
\begin{align}\label{qp}
\Lambda =\bigg\{(p, q)\in\R^2\bigg\vert \, 1<p< n,\quad \text{and}\quad q \,\geq \, 1 + \frac{(n-p)}{(p-1)(n-1)} \bigg\}
\end{align}
both $V_q^p$ and $V_{\infty}^p$ are monotone non-decreasing.
Moreover, the monotonicity is strict unless $\Omega$ is a ball, and, in particular, $u$ is rotationally symmetric. To be more precise, we state our main Monotonicity-Rigidity Theorem for the functions $V_q^p$, with $q < \infty$. 
\begin{theorem}
\label{mainth}
 Let $(p,q) \in \Lambda$ and let $u$ be a solution to \eqref{prob_ex}. Then $V_q^p$ is differentiable with derivative
\begin{equation}
\label{vqpder}
\frac{d{V_q^p}}{dt}(t)=(q-1)\left(\frac{{C_p}(\Omega)}{t^{p-1}}\right)^{\!\!\frac{(n-1)(q-1)}{(n-p)}}\!\!\!\!\!\!\int\limits_{\{u=t\}}\!\!\!\!\abs{Du}^{q(p-1)-1}\Big[H- \frac{(n-1)(p-1)}{(n-p)}\abs{D\log u}\Big]d\sigma.
\end{equation} 
where $H$ is the mean curvature of $\{u=t\}$ computed with respect to the unit normal vector $\nu=-Du/\abs{Du}$ .
For every $t\in (0, 1]$ such derivative satisfies
\begin{equation}
\label{quadratoni}
\begin{split}
\frac{d{V_q^p}}{dt}(t) = &{(q-1)}\!\!\left(\dfrac{{C_p(\Omega)}}{t^{p-1}}\right)^{\!\!\frac{(n-1)(q-1)}{n-p}}\!\! \!\!\!\!\!\int\limits_{\{u \geq t\}}\!\!\!\!\left(\frac{u}{t}\right)^{2- \frac{(n-1)(p-1)(q-1)}{(n-p)}} \abs{Du}^{q(p-1)-3} \,\, \times \\
&\times \bigg\{\left\vert D^2_T u - \frac{\Delta_T u}{n-1}g_T^{\R^n}\right\vert^2 
 + \Big(q(p-1) - 1\Big) \big\vert D_T \abs{Du}\big\vert^2 \\
&\,\,\,\,+\left[q - 1 - \frac{(n-p)}{(p-1)(n-1)}\right] \abs{Du}^2 \left[H - \frac{(n-1)(p-1)}{(n-p)} \abs{D \log u}\right]^2\bigg\} d\mu,
\end{split}
\end{equation}
where, for any $x \in \{u \geq t\}$ the function $H(x)$ is the mean curvature of $\{u = u(x)\}$ with respect to $\nu$, and the tangential elements are referred to these level sets (see Notation \ref{notation} below). In particular, the derivative of $V_q^p$ is always non-negative, and it vanishes for some $t\in (0, 1]$ if and only if $\Omega$ is a ball and $u$ is rotationally symmetric. 
\end{theorem}

\begin{notation}
\label{notation}
We explain here the meaning of the tangential elements appearing on the right hand side of formula \eqref{quadratoni}, as well as in the rest of the paper. For any $x$ in $\R^n \setminus \Omega$ we consider an orthonormal basis of $T_x\R^n$  of the form $\{e_1, \dots, e_{n-1}, e_n={Du}/\abs{Du}\}$. Consequently, we define
\[
{g_T^{\R^n}}_{\!\!\!\mid x} = \sum_{i=1}^{n-1} e^i \!\otimes e^i,
\]
where the raised indexes denote as usual the dual basis. For a function $f \in C^2(\R^n \setminus \Omega)$ we also define
\[
D_T f_{\mid x} = \sum_{i=1}^{n-1} (D_{e_i}f)_{\mid x}\, e^i,
\]
\[
D^2_T f_{\mid x} = \sum_{i, j = 1}^{n-1} D^2 f_{\mid x} (e_i, e_j)\, e^i\! \otimes e^j
\]
and finally
\[
\Delta_T f_{\mid x} = \Delta f_{\mid x} - D^2 f_{\mid x} (e_n, e_n).
\]
Notice in particular that $D^2_T f$ and $\Delta_T f$ must not be confused with the tangential Hessian and the tangential Laplacian induced by $g_T^{\R^n}$ on the level sets of $u$.
\end{notation}

A completely analogous Monotonicity-Rigidity Theorem can be stated for $V_\infty^p$. We point out that the monotonicity of $V_\infty^p$ is not completely new, since it is related to a maximum principle for the $P$-function contained in \cite{garofalosartori} (see Theorem 2.2, Theorem 3.1 and Lemma 5.1 therein). However, our proof is quite different, since it is  carried out in a conformal setting and inspired by Colding's work on monotonicity formulas for the Green's function of the Laplacian on Riemannian manifolds \cite{colding}.
\begin{theorem}
\label{mainthsup}
Let $u$ be a solution to \eqref{prob_ex}. Then, the following assertions hold true.
\begin{enumerate}
\item[i)] The function $V_\infty^p$ is monotone non-decreasing. Moreover, $V_\infty^p (t_1) = V_\infty^p(t_2)$ for some $t_1\neq t_2 \in (0, 1]$ if and only if $\Omega$ is a ball and $u$ is rotationally symmetric.
 \item[ii)]Let $x_t \in \{u=t\}$ be a maximum point of the  function $\abs{Du}/u^{(n-1)/(n-p)}$ on $\{u=t\}$. Then,
\begin{equation}
\label{supvder}
\left[H - \frac{(n-1)(p-1)}{(n-p)}\abs{D \log u}\right] (x_t) \geq 0,
\end{equation}
and equality is achieved for some $t \in (0, 1]$ if and only if $\Omega$ is a ball and $u$ is rotationally symmetric. Here $H$ denotes the mean curvature of $\{u=t\}$ computed with respect to the unit normal vector $\nu=-Du/\abs{Du}$.
\end{enumerate}
\end{theorem} 
\begin{figure}[b]
\centering
\includegraphics[keepaspectratio, height = 13cm]{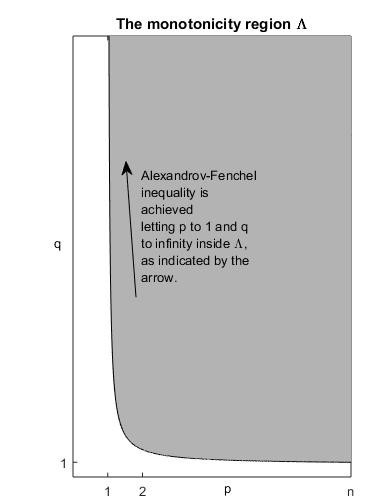}
\end{figure}
These facts imply a number of sharp estimates involving $u$ and $\Omega$, that can be 
gathered as capacity estimates, overdetermining Neumann conditions, Sphere Theorems and purely geometric inequalities. Except when explicitly indicated, these results are new for $p \neq 2$. They will be discussed in detail in Section 3. Dropping any attempt to be complete we observe that other results in the same spirit can be found for example in \cite{xiao,bianchini, garofalo, garofalosartori, magna1, magna2, poggesi} and reference therein. 
%More precisely, the monotonicity of $V_q^p$ for $q< \infty$ allows us to generalize to  $p \in (1, n)$ all the results provided in \cite{AgoMaz2}, while the monotonicity of $V_\infty^p$ extends to the non-linear case the results of \cite[borghini]{borghini}. In particular, we provide a unified approach to the type of estimates considered in these two works, under the additional hypothesis of convexity. Indeed, both the monotonicity of $V_q^p$ and and $V_\infty^p$ will be proved in the conformal setting provided by the \emph{cylindrical ansatz}, introduced in \cite{AgoMaz2}, while in\cite[borghini]{borghini} a different geometric approach named by the authors \emph{spherical ansatz} was considered. 

In particular, as a specific feature of the monotonicity given in Theorem \ref{mainth}, we obtain a new proof of the classical Minkowski inequality for smooth and convex domains, also known in literature as Alexandrov-Fenchel inequality.
\begin{theorem}[Alexandrov-Fenchel inequality]
\label{geometric1}
Let $\Omega\subset\R^n$ be a smooth, bounded and convex domain.
Then the following inequality holds:
\begin{equation}
\label{geoeq1}
\left(\frac{\abs{\Sf^{n-1}}}{\abs{\partial \Omega}}\right)^{\!\frac{1}{n-1}}\!\!\leq \fint\limits_{\partial\Omega}{\frac{H}{n-1}} d\sigma,
\end{equation}
where $H$ is the mean curvature of $\partial \Omega$ computed with respect to the exterior unit normal.
\end{theorem}

The above geometric inequality is deduced from Theorem \ref{mainth} roughly as follows. The global monotonicity of $V_q^p$ implies $\lim_{t \to 0^+} V_q^p (t) \leq V_q^p (1)$, and then, one deduces, using the asymptotics of $u$, that
\[
\abs{\Sf^{n-1}}^{\frac{1}{q(p-1)}}\left({{C_p}(\Omega)^{\frac{1}{p-1}}}\right)^{1-\frac{(n-1)(q-1)}{(n-p)q}}\leq\bignorm{\frac{p-1}{n-p}D(\log u)}_{L^{q(p-1)}(\partial \Omega)}.
\]
At this point, from \eqref{vqpder} and the H\"older inequality one obtains
\[
\frac{\abs{\Sf^{n-1}}}{\abs{\partial\Omega}}\leq{C_p}(\Omega)^{\frac{(p-1)q-(n-1)}{(n-p)}}\fint\limits_{\partial\Omega}\left({\frac{H}{n-1}}\right)^{q(p-1)}d\sigma.
\] 
Finally, Theorem \ref{geometric1} is proved simply by plugging $q = {p}/(p-1)$ in the above estimate, passing to the limit as $p \to 1^+$ and taking into account the fact that $C_1(\Omega) = \abs{\partial \Omega}/{\abs{\Sf^{n-1}}}$.

Alexandrov-Fenchel inequalities for convex domains were introduced in \cite{alex1} and \cite{alex2}, and in \cite{fenchel} for closed curves. Since then, they have been thoroughly studied. It is nowadays well known that they can be deduced using the Inverse Mean Curvature Flow (IMCF for short). Let us quote some of the main achievements of this technique in the present context. In \cite{pengfeiguan}, the IMCF has been used to obtain the Alexandrov-Fenchel inequality assuming that $\Omega$ is star-shaped and $\partial \Omega$ is mean convex.  
Such a procedure relies on the important works by Gerhardt (\cite{gerry}) and Urbas (\cite{urbas}), where they prove that under the above assumptions on $\Omega$ the solutions to the IMCF are defined for all positive times and approach Euclidean spheres as the time tends to infinity.
In \cite{huisken} and \cite{freire}, the Minkowski inequality is proved assuming that $\partial \Omega$ is outward minimizing, a property that implies mean convexity but not related to star-shapedness. It is important to notice that unlike the ones used in \cite{pengfeiguan}, this hypothesis does not force $\partial \Omega$ to have the topology of the sphere. Actually, a crucial step in this circle of ideas consists in showing that the monotonity of a suitable quantity is preserved also through singularities of the flow, that necessarily happen in the topologically non-spherical case. This is achieved by means of the techniques introduced in the celebrated work by Huisken and Ilmanen \cite{huisken-ilmanen} where a suitable notion of weak solutions for the IMCF is defined. 
Let us finally close this short excursus mentioning  \cite{chang}, where Alexandrov-Fenchel-type inequalities  are proved by means of optimal transport methods.

Of course our convexity assumption, widely used in literature  to study problems similar to \eqref{prob_ex} (see \cite{Ark,CNSXYZ,CS2} and references therein), is stronger than the ones discussed above. However, here we are focusing on a new self-contained method. 
A possible advantage of it is that the level set flow we are employing exists for any time and for any bounded $\Omega$.
Indeed, there always exists a weak solution to the exterior problem \eqref{prob_ex}, and this solution can be proved to be $C^{1, \alpha}_{\text{loc}}$.  For these regularity results, see \cite{dibenedetto, lewis-reg} for $1 < p < 2$, \cite{uhlenbeck, uraltseva, evans} for $p > 2$. 
In particular, our flow would make sense even without a mean convexity assumption, that is instead necessary in order to let the IMCF start. 
We introduce convexity because it ensures that all the level sets of the potential we are considering are regular, and thus we can work out the smooth theory. Indeed, as already remarked, this is the content of a famous result of Lewis (\cite{Lew}, Theorem \ref{lewis} below). In order to put the work in perspective, a possible future development consists in studying weaker assumptions on $\Omega$ providing the desired regularity, in the spirit of Gerhardt and Urbas' works. Indeed, we stress the fact that our theory would work under any of these conditions on $\Omega$. In this direction, see the nice paper \cite{peralta} where the harmonic setting is considered. Another, and more ambitious possibility, would be establishing our new monotonicity also through singularities, in the wake of Huisken-Ilmanen's techniques. Relations between $p$-harmonic functions and weak solutions to the IMCF have already been considered in the remarkable paper \cite{moser}.
Let us finally point out that, in the harmonic theory proposed in \cite{AgoMaz2}, the well known upper bounds on the Hausdorff and Minkowski dimensions of the critical set of harmonic functions allowed the authors to establish monotonicity formulas for any bounded and smooth $\Omega$. For a proof of these regularity results in the much more general setting of linear, homogeneous, second-order elliptic equation, we refer the reader to \cite{cheeger}. On the contrary, in the non-linear setting, the structure of the critical set is still an open problem, at least to the authors' knowledge. 
\newline{}

The paper is organized as follows. In Section 2, after some preliminaries on problem \eqref{prob_ex} and our monotone quantities, we state and prove all the consequences of Theorem \ref{mainth} and Theorem \ref{mainthsup}.
These consequences will be  divided in local consequences (roughly speaking, expoiting $(V_q^p )'(1) \geq  0$) and global consequences (exploiting $\lim_{t \to 0^+} V_q^p(t) \leq V_q^p(1)$). Theorem \ref{geometric1} will follow by this second procedure. In Section 3 we introduce a conformally equivalent formulation of our problem. More precisely, we are going to the describe the \emph{cylindrical ansatz} in the non-linear setting, and to state the conformal version of our Monotonicity-Rigidity Theorems.
Section 4 is devoted to the proof of (the conformal version of) Theorem \ref{mainth}. Among the various steps, some  sharp inequalities for $p$-harmonic functions on Riemannian manifolds are proved, of interest in themselves.  Finally, in Section 5 we prove (the conformal version of) Theorem \ref{mainthsup}. 

\section{Preliminaries and consequences of the main results}
\subsection{Preliminaries} The classical theorem of Lewis, proved in \cite{Lew}, essentially provides all the background we need to apply our methods. We are going to state it essentially as reported in \cite{CNSXYZ}.
\begin{theorem}
\label{lewis}
Let $n>2$, $1<p<n$ and let $\Omega\subset\mathbb{R}^n$ be a bounded convex domain. Then there exists a unique weak solution $u$ to \eqref{prob_ex} satisfying the following
\begin{itemize}
	\item[(i)] $u\in C^{\infty}(\mathbb{R}^n\setminus \overline{\Omega})\cap C(\mathbb{R}^n\setminus \Omega)$;
	\item[(ii)] $0<u<1$ and $|D u|\neq 0$ in $\mathbb{R}\setminus\overline{\Omega}$;
	\item[(iii)] Let $C_p(\Omega)$ be the rescaled $p$-capacity of $\Omega$ defined by \eqref{capacity}.
Then 
\begin{equation}
\label{cpf1}
C_p(\Omega)=\frac{1}{\Big(\frac{n-p}{p-1}\Big)^{p-1}\abs{\Sf^{n-1}}}\int_{\mathbb{R}^n\setminus\overline{\Omega}} |D u|^p d\mu;
\end{equation}
	\item[(iv)] If $u$ is defined to be $1$ in $\Omega$, then 
	\[\Omega_t=\{x\in\mathbb{R}^n\ |\ u(x)>t\}\]
	is convex for each $t\in [0,1]$ and $\partial\Omega_t$ is a $C^{\infty}$ manifold for $0<t<1$.
\end{itemize}
\end{theorem}
We are going to use the following well known expression for $C_p(\Omega)$ in terms of an integral on $\partial \Omega$.
\begin{lemma}
\label{cplemma}
Let $u$ be the  solution to \eqref{prob_ex}. Then
\begin{equation}
\label{capf}
C_p(\Omega)=\dfrac{\bigintssss\limits_{\partial\Omega}\abs{Du}^{p-1}d\sigma}{\left(\frac{n-p}{p-1}\right)^{p-1}\abs{\Sf^{n-1}}}.
\end{equation}
\end{lemma}
\begin{proof}
By exploiting the $p$-harmonicity of $u$ and the Divergence theorem we have
\[
0=\int_{\{t<u<1\}}\Delta_p u\ dx= \int_{\{u=t\}}\abs{Du}^{p-1}d\sigma - \int_{\partial \Omega}\abs{Du}^{p-1}d\sigma
\]
which implies 
\[
\int_{\{u=t\}}\abs{Du}^{p-1}d\sigma=\int_{\partial \Omega}\abs{Du}^{p-1}d\sigma.
\]
Thus, by co-area formula and \eqref{cpf1}, we have
\[
\left(\frac{n-p}{p-1}\right)^{p-1}\abs{\Sf^{n-1}}C_p(\Omega)= \int_{0}^1\int_{\{u=t\}}\abs{Du}^{p-1}d\sigma=\int_{\partial \Omega}\abs{Du}^{p-1}d\sigma.
\]
\end{proof}

\begin{remark}[$u$ is analytic]
\label{elliptic}
Condition \emph{(b)} in the above theorem actually implies analyticity of the solution $u$, by an application of local regularity theory developed in \cite{lady}. The same observation was crucial also in the proof of \cite[Theorem 2.4]{garofalosartori}.
\end{remark}

The following asymptotics for $u$  will be important later on to compute the limits of our monotone quantities.
\begin{proposition}[\cite{CNSXYZ}, Lemma 2.15]
\label{mono_and}
Let $n>2$. Suppose $1<p<n$ and let $\Omega\subset \mathbb{R}^n$ be a bounded, smooth convex domain. 
If $u$ is a solution to \eqref{prob_ex}, then 
\begin{itemize}
	\item[(i)] $\lim_{|x|\to\infty} u(x)|x|^{\frac{n-p}{p-1}}=C_p(\Omega)^{\frac{1}{p-1}}$.
	\item[(ii)] $\lim_{|x|\to\infty} |D u(x)||x|^{\frac{n-1}{p-1}}=C_p(\Omega)^{\frac{1}{p-1}}\big(\frac{n-p}{p-1}\big)$.
\end{itemize}
\end{proposition}

We recall here for the reader's convenience the definition of $V_q^p: (0, 1] \mapsto \R$ given in the Introduction:
\begin{equation*}
\label{defvqpSSSSSSS}
V_q^p(t)=
\begin{system}
   &\left(\dfrac{{C_p(\Omega)}}{t^{p-1}}\right)^{\frac{(n-1)(q-1)}{n-p}}\!\!\!\!\!\!\!\!\!\int_{\{u=t\}}\abs{Du}^{q(p-1)} \,d\sigma& \qquad \qquad &\mbox{if}&\qquad 0\leq q <\infty\\
	 & \sup\limits_{{\{u=t}\}} \dfrac{\abs{Du}}{u^{\frac{n-1}{n-p}}} & \qquad &\mbox{if}&\qquad q=\infty.
\end{system}
\end{equation*}

\begin{remark}
\label{finite}
We point out that $V_q^p(t)<\infty$ for any $t \in (0, 1]$ and for any $0\leq q\leq \infty$. Indeed, $\abs{Du}$ is a continuous function and by Theorem \ref{lewis} \emph{(iv)} and the asymptotics given in Proposition \ref{mono_and} \emph{(i)}, $\{u=t\}$ is a smooth compact set. Moreover, $\abs{Du} > 0$ by (ii) of Theorem \ref{lewis}, and then  also the integral in \eqref{vqpder} is finite for any $(p, q) \in \Lambda$.
%We also point out that \eqref{vqpder} extends \cite[Theorem 1.3]{bianchini} to $(p, q) \in \Lambda$.
\end{remark}

By the asymptotics given in Proposition \ref{mono_and}, it is just a matter of straightforward computation obtaining the limits of our monotone quantities. 

\begin{lemma}[Limits of $V_q^p$]
\label{limitlemma}
Let $V_q^p : (0, 1] \to \R$ be defined as in \eqref{defvqp}. Then
\begin{equation}
\label{limvqp1}
\lim_{t \to 0^+}V_q^p(t)={C_p}(\Omega)^{{q}}\left(\frac{n-p}{p-1}\right)^{q(p-1)}\abs{\Sf^{n-1}}, \qquad \text{if} \qquad q <\infty
\end{equation}
and
\begin{equation}
\label{limvqinf}
\lim_{t \to 0^+} V_\infty^p(t)=\left(\frac{n-p}{p-1}\right) C_p(\Omega)^{-\frac{1}{n-p}} ,
\end{equation}
\end{lemma}

\subsection{Consequences of the main Theorems} 

In this section we mainly follow the scheme proposed in \cite{AgoMaz2} to get various consequences of Theorem \ref{mainth} and Theorem \ref{mainthsup}. More precisely, in the first subsection we use \eqref{vqpder} and \eqref{supvder}, to deduce various sharp inequalities involving $u$ and $\Omega$, while in the second subsection we will compare the value of our monotone functions on different level sets of $u$. The most interesting byproducts of this theory will arise by this second procedure. 
In other words, we are going to exploit both the local and the global features of our theorems. 

Let us also point out that the monotonicity of $V_q^p$ allows us to extend to the non-linear case all the results provided in \cite{AgoMaz2}, while the monotonicity of $V_\infty^p$ extends the results contained in \cite{borghini}.  In particular, we provide a unified approach to the type of estimates considered in these two papers. Indeed, both the monotonicity of $V_q^p$ and and $V_\infty^p$ will be proved in the conformal setting provided by the \emph{cylindrical ansatz}, introduced in \cite{AgoMaz2}, while in \cite{borghini} a different geometric approach named by the authors \emph{spherical ansatz} was considered.
Before proceeding notice that $\partial \Omega = \{u=1\}$, and consequently, if $\nu$ is the (interior) normal to $\partial\Omega$, we get $\partial u/ \partial \nu = -\abs{Du}$.

\subsubsection{Local consequences}

We first notice that a direct consequence of \eqref{vqpder} or \eqref{supvder} is an overdetermining Neumann condition for the exterior problem forcing the solution to be rotationally symmetric. 
\begin{corollary}
\label{over}
Let $u$ be a solution to problem \eqref{prob_ex} and assume that the identity 
\[
\frac{p-1}{n-p}\abs{Du}=\frac{H}{n-1}
\]
holds $\Haus{n-1}$-almost everywhere on $\partial\Omega$, where $H$ is the mean curvature of $\partial\Omega$. Then $u$ is rotationally symmetric. In particular, if $u$ solves the overdeterminated boundary value problem
\begin{align}
\left\{\begin{array}{lll}\label{profc}
\Delta_p u=0 & \mbox{in} & \R^n\setminus \overline{\Omega}\\
\quad \,\,u=1 & {\text{on}} & \partial\Omega\\
\quad\!\frac{\partial u}{\partial\nu}=-\Big(\frac{n-p}{(p-1)(n-1)}\Big)H & \mbox{on} & \partial\Omega \\
\,u(x)\to 0 & \mbox{as} & |x|\to\infty\\
\end{array}\right.
\end{align}
where $\nu$ is the unit normal vector to $\partial\Omega$ pointing toward the interior of $\R^n\setminus\overline{\Omega}$, then $\Omega$ is a ball and $u$ is rotationally symmetric.
\end{corollary}
We are now going to improve Corollary \ref{over} in several different ways.
The following result involves just Lebesgue norms of the normal derivative of $u$ on $\partial \Omega$ .

\begin{theorem}
\label{corder}
Let $(p, q)\in\Lambda$ and $u$ be a solution to \eqref{prob_ex}. Then it holds
\begin{equation}
\label{lpder}
\|Du\|_ {L^{(p-1)q}(\partial\Omega)}= \bignorm{\frac{\partial u}{\partial\nu}}_{L^{(p-1)q}(\partial\Omega)}\leq\left(\frac{n-p}{(p-1)(n-1)}\right)\norm{H}_{L^{(p-1)q}(\partial\Omega)},
\end{equation}
and
\begin{equation}
\label{linftyder}
\sup_{\partial \Omega} |Du|=\sup_{\partial \Omega}\bigabs{\frac{\partial u}{\partial \nu}} \leq \sup_{\partial \Omega}\left(\frac{n-p}{(p-1)(n-1)}\right)H
\end{equation}
Moreover, equality holds in \eqref{lpder} or in \eqref{linftyder} if only if $\Omega$ is a ball and $u$ is rotationally symmetric. 
\end{theorem}
\begin{proof}
Plugging $t=1$ into \eqref{vqpder} and recalling that $|D\log u|=|Du|$ in $\{u=1\}$ we get
\begin{equation}
\label{cons1f}
\int_{\partial \Omega}\frac{p-1}{n-p}{\abs{D\log u}}^{(p-1)q}d\sigma\leq\int_{\partial \Omega}\abs{D\log u}^{(p-1)q-1}\frac{H}{n-1}d\sigma.
\end{equation}
Applying H\"older inequality to \eqref{cons1f} gives \eqref{lpder}. The rigidity part of the statement follows from the related part in Theorem \ref{mainth}.

Inequality \eqref{linftyder} together with the rigidity statement follows applying \eqref{supvder} with $t=1$ and recalling that $x_1\in \{u=1\}$ is the maximum point of $|Du|$ on $\{u=1\}$.
\end{proof}

%Clearly, the overdetermining condition on $u$ given by Corollary \ref{corder} are weaker than the one considered in Theorem \ref{over}.
%\newline{}

Using \eqref{lpder} and \eqref{linftyder} we can then easily prove the following geometric estimates for $C_p(\Omega)$.
\begin{theorem}
\label{cpth}
Let $\Omega\subset\R^{n}$ be a bounded convex domain with smooth boundary. Let $(p, q)\in\Lambda$. Then
\begin{equation}
\label{cpthf}
C_p(\Omega)\leq\frac{\abs{\partial\Omega}}{\abs{\Sf^{n-1}}}\left({\fint_{\partial\Omega}\left(\frac{H}{n-1}\right)}^{{(p-1)}q}d\sigma\right)^{\frac{1}{q}}
\end{equation}
and
\begin{equation}
\label{cpthinfty}
C_p(\Omega)\leq \frac{\abs{\partial \Omega}}{\abs{\Sf^{n-1}}}\sup_{\partial \Omega}\left({\frac{H}{n-1}}\right)^{p-1}.
\end{equation}
Moreover, equality is achieved in \eqref{cpthf} or in \eqref{cpthinfty}  if and only if $\Omega$ is a ball.
\end{theorem}
\begin{proof}
By H\"older inequality with conjugate exponents $q$ and $q/(q-1)$ we obtain
\[
{\left(\int_{\partial\Omega}\abs{Du}^{p-1}d\sigma\right)}^{\frac{1}{p-1}}\leq\norm{Du}_{L^{(p-1)q}(\partial\Omega)}\abs{\partial\Omega}^{\frac{q-1}{q(p-1)}}
\]
and \eqref{cpthf} follow by \eqref{lpder} and  \eqref{capf}. Inequality \eqref{cpthinfty} can be easily proved observing that
\[
\int_{\partial \Omega} \abs{Du}^{p-1}\ d\sigma \leq \sup_{\partial \Omega}\abs{Du}^{p-1} \abs{\partial \Omega},
\]
and applying \eqref{linftyder} and \eqref{capf}. The rigidity follows from the rigidity part of Corollary \ref{corder}.
\end{proof}

\begin{remark}
\label{indip}
Since for every $p \in (1, n)$  the couple $(p, (n-1)/(p-1))$ belongs to $\Lambda$, inequality \eqref{cpthf} yields the following $p$-independent estimate for the $p$-capacity
\[
C_p(\Omega)^{\frac{1}{p-1}} \leq \frac{\abs{\partial\Omega}}{\abs{\Sf^{n-1}}}\left({\fint_{\partial\Omega}\left(\frac{H}{n-1}\right)}^{n-1}d\sigma\right)^{\frac{1}{n-1}}.
\] 
together with a rigidity statement when equality is attained. 
The integral appearing in the right hand side of the above inequality is known in literature as Willmore functional. 
We point out that the previous inequality has been proved in \cite{xiao} under weaker assumptions on $\Omega$ and by different methods.
\end{remark} 
 
\begin{remark}[A classical overdetermined problem] Consider the problem of characterizing bounded domains $\Omega \subset \R^n$ supporting a solution to the classical overdetermined exterior problem  for the $p$-Laplace operator:
\begin{align}\label{serrin_pb}
\left\{\begin{array}{lll}
\Delta_p u=0 & \mbox{in} & \R^n\setminus \overline{\Omega}\\
\quad\,\, u=1 & \mbox{on} & \partial\Omega\\
\quad\!\frac{\partial u}{\partial \nu}= -c  & \mbox{on} & \partial \Omega \\
\, u(x)\to 0 & \mbox{{as}} & |x|\to\infty,
\end{array}\right.
\end{align}
where $c$ is a positive constant.
Then, in this setting, Theorem \ref{mainthsup} can be easily combined with the techniques used in \cite{garofalosartori} to obtain the rotational symmetry of $\Omega$ for bounded and convex domains.  In this regard, observe that the monotonicity of $V_\infty^q$ readily implies a maximum principle
stating that, for any $t \in (0, 1]$
\[
\sup_{\{u\leq t\}} \frac{\abs{Du}}{u^{\frac{n-1}{n-p}}}=\sup_{\{u = t\}} \frac{\abs{Du}}{u^{\frac{n-1}{n-p}}}.
\] 
This type of estimate was actually the key ingredient in Garofalo-Sartori's arguments.
Let us finally also point out that symmetry for this type of overdetermined problems has been established in much more generality by Reichel in \cite{reichel}.
\end{remark}
\subsubsection{Global consequences}

We turn our attention to the global features of our monotonicity theorems. 
We consider separately $V_q^p$ and $V_\infty^p$.
\newline{}

Let first $q < \infty$. Since $V_q^p$ is non-decreasing we have
\begin{equation}
\label{global}
\lim_{t\to 0^+}V_q^p(t)\leq V_q^p(1).
\end{equation}
By Lemma \ref{limitlemma} 
\begin{equation}
\label{limvqp}
\lim_{t\to 0^+}V_q^p(t)={C_p}(\Omega)^{{q}}\left(\frac{n-p}{p-1}\right)^{q(p-1)}\abs{\Sf^{n-1}},
\end{equation}
Inserting the above expression into \eqref{global}, some elementary algebra and \eqref{lpder} give the following inequalities
\[
\small
\abs{\Sf^{n-1}}^{\frac{1}{q(p-1)}}\left({{C_p}(\Omega)^{\frac{1}{p-1}}}\right)^{1-\frac{(n-1)(q-1)}{q(n-p)}}\!\!\!\!\leq\bignorm{\frac{p-1}{n-p}D(\log u)}_{L^{q(p-1)}(\partial \Omega)}\leq\bignorm{\frac{H}{n-1}}_{L^{q(p-1)}(\partial \Omega)},
\]
for every $(p,q)\in\Lambda$.
Equalities  in the chain above are achieved  if and only if $\Omega$ is a ball and $u$ is rotationally symmetric. Rearranging the terms we are left with the following estimate for the $p$-capacity of $\Omega$. 
\begin{theorem}
\label{capgl}
Let $\Omega\subset\R^n$ be a bounded, convex and smooth domain, and let $(p, q) \in \Lambda$. Then
\begin{equation}
\label{capf1}
\frac{\abs{\Sf^{n-1}}}{\abs{\partial\Omega}}\leq[{C_p}(\Omega)]^{\frac{q(p-1)-(n-1)}{(n-p)}}\fint\limits_{\partial\Omega}\left(\frac{H}{n-1}\right)^{q(p-1)}.
\end{equation}
Moreover, the equality is achieved if and only if $\Omega$ is a ball.
\end{theorem}

%We now use \eqref{capf1} in order to gain information about the quantity
%\[
%\cE_p(\Omega):=\frac{C_p(\Omega)}{\Big(\frac{\abs{\partial\Omega}}{\abs{\Sf^{n-1}}}\Big)^{\frac{n-p}{n-1}}}.
%\]
%The above definition is related to a conjecture of P\'olya-Szeg\"o, stated in the case $n=3$ and $p=2$, see \cite{gazzola} for an overview on this subject.
%Choosing parameters $q$ and $r$ respectively smaller and greater than $(n-1)/(p-1)$, we obtain the following result.
%\begin{corollary}
%Fix $1<p<n$, let $q, r \in \R$ such that $(p, q) \in \Lambda$, $(p, r) \in \Lambda$ and
%\[
%(p-1)q<{n-1}<(p-1)r.
%\]
%We have
%\begin{equation}
%\label{polya}
%\left[\frac{\Big(\frac{\abs{\Sf^{n-1}}}{\abs{\partial\Omega}}\Big)^{\frac{p-1}{n-1}}}{\Big(\fint_{\partial\Omega}{\left(\frac{H}{n-1}\right)}^{r(p-1)}\Big)^{1/r}}\right]^{\frac{r(n-p)}{r(p-1)-(n-1)}}\leq\frac{C_p(\Omega)}{\Big(\frac{\abs{\partial\Omega}}{\abs{\Sf^{n-1}}}\Big)^{\frac{n-p}{n-1}}}\leq\left[\frac{\Big(\fint_{\partial\Omega}{\left(\frac{H}{n-1}\right)}^{q(p-1)}\Big)^{1/q}}{\left(\frac{\abs{\Sf^{n-1}}}{\abs{\partial\Omega}}\right)^{\frac{p-1}{n-1}}}\right]^{\frac{q(n-p)}{(n-1)-q(p-1)}}.
%\end{equation}
%Moreover, equality holds for some admissible $p, q$ and $r$ if %and only if $\Omega$ is a ball.
%\end{corollary}

Let us come back to inequality \eqref{capf1}. We are going to deduce two purely geometric consequences of this sharp estimates.  Choosing parameters such that $q(p-1)=n-1$, the term involving the capacity disappears 
and we are left with the classical Willmore inequality \cite{willmore} together with its rigidity statement. Recall that, as observed in Remark \ref{indip}, such a choice of parameters $(p, q)\in\Lambda$  is possible for any $p \in (1, n)$. 
\begin{corollary}[Willmore inequality]
Let $\Omega\subset\R^n$ be a bounded, convex domain with smooth boundary. Then
\[
\abs{\Sf^{n-1}}\leq \int\limits_{\partial\Omega}{\left(\frac{H}{n-1}\right)}^{n-1}d\sigma.
\]
Equality holds if and only if $\Omega$ is a ball.
\end{corollary}

\begin{remark}
As shown in \cite{AgoMaz2}, when $p=2$ it is not difficult to get rid of the convexity assumption, and thus, to obtain Willmore inequality on any bounded domain with smooth boundary.
However, being able to obtain Willmore inequality via a solution to \eqref{prob_ex} for any $p \in (1, n)$ could have some interest when trying to adapt our techniques to Riemannian manifolds supporting a solution to \eqref{prob_ex}, known in literature as $p$-nonparabolic, or $p$-hyperbolic manifolds. This could be object of future works. 
\end{remark}

Differently from Willmore inequality, the Alexandrov-Fenchel inequality  is a specific feature of Theorem \ref{mainth} for $p \neq 2$. In fact, we are going to pass to the limit in \eqref{capf1}  simultaneously as $p \to 1$ and $q \to \infty$, as sketched in the Introduction. 
\begin{proof}[Proof of Theorem \ref{geometric1}]
%\begin{theorem}[General Alexandrov-Fenchel inequality]
%\label{geometric}
%Let $\Omega\subset\R^n$ be a smooth, bounded and convex set.
%Then, for any $\alpha\geq 1$, the following inequality holds:
%\begin{equation}
%\label{geoeq}
%\frac{\abs{\Sf^{n-1}}}{\abs{\partial \Omega}}\leq \left(\int_{\partial\Omega}{\left(\frac{H}{n-1}\right)}^\alpha d\sigma\right)^{\frac{n-1}{\alpha}}.
%\end{equation}
%\end{theorem}
%\begin{proof}
Consider a sequence $p_m \to 1^{+}$, and let $q_m =  \frac{p_m}{p_m - 1}$. Obviously $(p_m, q_m)\in\Lambda$ for every $m\in\mathbb{N}$. Plugging $(p, q) = (p_m, q_m)$ into \eqref{capf1} and letting $m \to \infty$ we get
\[
\frac{\abs{\Sf^{n-1}}}{\abs{\partial\Omega}}\leq\big[C_1(\Omega)\big]^{\frac{1 - (n-1)}{(n-1)}}\fint_{\partial\Omega}\frac{H}{n-1} d\sigma,
\]
where we used that 
\[
\lim_{p \to 1^{+}} C_p(\Omega)= C_1 (\Omega)
\]
as proved in \cite[Theorem 11]{meyers}. 
Finally, since by \cite[Lemma 2.2.5]{mazia}  
\[
C_1(\Omega)=\frac{\abs{\partial\Omega}}{\abs{\Sf^{n-1}}},
\]
\newline{}
a simple rearrangement of terms ends the proof. 
%When $\Omega$ is a ball, \eqref{geoeq} becomes an equality, showing in fact the sharpness of this relation.
\end{proof}

%\begin{remark}
%Unfortunately, our procedure does not provide a rigidity statement if equality is achieved in  \eqref{geoeq}.
%\end{remark}

In analogy with what had just been done, we are going to exploit the fact that, due to Theorem \ref{mainthsup},
\begin{equation}
\label{limsup}
\lim_{\tau \to 0^+} \sup_{\{u=\tau\}}\frac{\abs{Du}}{u^{\frac{n-1}{n-p}}} \leq  \sup_{\{u=1\}}\frac{\abs{Du}}{u^{\frac{n-1}{n-p}}}= \sup_{\partial \Omega}{\abs{Du}}.
\end{equation}
The limit on the left hand side of the above inequality was computed in Lemma \ref{limitlemma} as 
\[
\lim_{\tau \to 0^+} \sup_{\{u=\tau\}}\frac{\abs{Du}}{u^{\frac{n-1}{n-p}}}=\left(\frac{n-p}{p-1}\right) C_p(\Omega)^{-\frac{1}{n-p}} .
\]
Thus, inequality \eqref{limsup} together with \eqref{linftyder} immediately yields the following result. 
\begin{theorem}
\label{globinf}
Let $\Omega \subset \R^n$ be a smooth, bounded and convex set, and let $u$ be a solution to \eqref{prob_ex}. Then the following chain of inequalities holds true.
\begin{equation}
\label{chainsup}
\left(\frac{1}{C_p(\Omega)}\right)^{\frac{1}{n-p}} \leq \sup_{\partial \Omega} \frac{p-1}{n-p} \abs{Du} \leq \sup_{\partial \Omega} \frac{H}{n-1}.
\end{equation}
Moreover, equality is achieved in one of the above inequalities if and only if $\Omega$ is a ball and $u$ is rotationally symmetric.
\end{theorem}

The above theorem has the following consequence in the framework of overdetermined boundary problems.
\begin{corollary}
\label{pinch}
Let $\Omega \subset \R^n$ be a smooth, bounded and convex set, and let $u$ be a solution to \eqref{prob_ex}. Assume that the (interior) normal derivative of $u$ on $\partial \Omega$ satisfies
\begin{equation}
\label{condpinch}
\bigabs{\frac{\partial u}{\partial \nu}} \leq \frac{n-p}{p-1} \left( \frac{\abs{\Sf^{n-1}}}{\abs{\partial \Omega}}\right)^{\frac{1}{n-1}},
\end{equation}
then $\Omega$ is a ball and $u$ is rotationally symmetric.
\end{corollary}
\begin{proof}
By \eqref{capf}, and assuming \eqref{condpinch},
we obtain the inequality
\[
C_p(\Omega) =\frac{1}{\left(\frac{n-p}{p-1}\right)^{p-1}\abs{\Sf^{n-1}}}\int_{\partial \Omega} \abs{Du}^{p-1} d\sigma \leq \left(\frac{\abs{\Sf^{n-1}}}{\abs{\partial \Omega}}\right)^{-\frac{n-p}{n-1}}.
\]
The above inequality together with the first inequality in \eqref{chainsup} implies
\[
\left(\frac{\abs{\Sf^{n-1}}}{\abs{\partial \Omega}}\right)^{\frac{n-p}{n-1}} \leq \left(\frac{1}{C_p(\Omega)}\right) \leq \left(\sup_{\partial {\Omega}} \frac{p-1}{n-p} \abs{Du}\right)^{n-p} \leq \left(\frac{\abs{\Sf^{n-1}}}{\abs{\partial \Omega}}\right)^{\frac{n-p}{n-1}}.
\]
In particular, equality must occur in the above chain of inequalities, and thus the rigidity part of Theorem \ref{globinf} allows to conclude.
\end{proof}

We conclude rephrasing Theorem \ref{globinf} as a sphere theorem under a pinching condition on the mean curvature of $\partial \Omega$, see also \cite{borghini}.

\begin{corollary}[Sphere Theorem]
Let $\Omega \subset \R^n$ be a smooth, bounded and convex set, and let $u$ be a solution to \eqref{prob_ex}. If the mean curvature $H$ of $\partial \Omega$ satisfies
\[
\frac{H}{n-1}\leq \left(\frac{1}{C_p(\Omega)}\right)^{\frac{1}{n-p}},
\]
then $\Omega$ is a ball.
\end{corollary}

\section{Conformal setting}
\subsection{A conformally equivalent formulation of the problem.}
We present an equivalent formulation of problem \eqref{prob_ex} which is based on a conformal change of the Euclidean metric. We set up the notation that we will use for all the rest of the paper. We first let
\begin{equation}
\label{m}
M := \R^n \setminus \Omega.
\end{equation}
 We denote by $g_{\R^n}$ the standard flat Euclidean metric of $\R^n$ and we consider, for the solution $u$ of \eqref{prob_ex}, the conformally equivalent metric given by
\begin{equation}\label{defmet}
g:=u^{2\frac{p-1}{n-p}}g_{\R^n}.
\end{equation}
Finally,
\begin{equation}\label{defpsi}
\psi:=-\frac{n-2}{n-p}(p-1)\log u
\end{equation}
(note that $\psi>0$), so that the metric $g$ can be equivalently written as 
\begin{equation*} g=e^{-\frac{2\psi}{n-2}}g_{\R^n}.
\end{equation*}
Fixing local coordinates $\{x^{\alpha}\}_{\alpha=1}^n$ in $M$ and using standard formulas \cite{HE} (see also \cite{AgoMaz1}) we get
\begin{align}\label{cambioCr}
\Gamma^{\gamma}_{\alpha\beta}&=G^{\gamma}_{\alpha\beta}-\frac{1}{n-2}\big(\delta^{\gamma}_{\alpha}\partial_{\beta} \psi+\delta^{\gamma}_{\beta}\partial_{\alpha} \psi-g_{\alpha\beta}^{\R^n}g_{\R^n}^{\gamma\eta}\partial_{\eta}\psi\big)\\
\label{cambioRic}
R_{\alpha\beta}^{g}&=R^{\R^n}_{\alpha\beta}+D_{\alpha}D_{\beta}\psi+\frac{\partial_{\alpha}\psi\partial_{\beta}\psi}{n-2}-\frac{|D\psi|^2-\Delta\psi}{n-2}g_{\alpha\beta}^{\R^n},\\
\label{cambiocon}
\nabla_{\alpha}\nabla_{\beta}w&=D_{\alpha}D_{\beta}w+\frac{1}{n-2}\Big(\partial_{\alpha}w\partial_{\beta}\psi+\partial_{\alpha}\psi\partial_{\beta}w-\left\langle Dw, D\psi\right\rangle g_{\alpha\beta}^{\R^n}\Big)\quad \forall w\in C^2(M),\\
\label{cambioLap}
\Delta_g w&=e^{\frac{2\psi}{n-2}}\left(\Delta w-\left\langle Dw,D\psi\right\rangle\right)\quad \forall w\in C^2(M).
\end{align}
 where $\Gamma^{\gamma}_{\alpha\beta}$ and $G^{\gamma}_{\alpha\beta}$ are the Christoffel symbols associated to the metric $g$ and $g_{\R^n}$ respectively, $R_{\alpha\beta}^{g}$ and $R^{\R^n}_{\alpha\beta}$ are the components of the Ricci tensor with respect to the metric $g$ and $g_{\R^n}$ respectively and $\nabla_{\alpha}$ and $D_{\alpha}$ are the covariant derivatives of the metric $g$ and $g_{\R^n}$ respectively. Notice that throughout this paper the Einstein summation convention for the sum over repeated indices is adopted. 
Let $X$ a vector field. Therefore, 
\begin{equation*}
\dive_{g}(X)=g^{ik}\big(\frac{\partial X_k}{\partial x_i}-\Gamma^l_{ik}X_l\big)=u^{-2\frac{p-1}{n-p}}g_{\R^n}^{ik}\big(\frac{\partial X_k}{\partial x_i}-\Gamma^l_{ik}X_l\big).
\end{equation*}
Using \eqref{cambioCr}, we get
\begin{equation*}
\Gamma^l_{ik}=G^l_{ik}+\frac{p-1}{n-p}\Big(\delta^{l}_{i}\frac{\partial_{k} u}{u}+\delta^{l}_{k}\frac{\partial_{i} u}{u}-g_{ik}^{\R^n}\frac{D^l u}{u}\Big)
\end{equation*}
and
\begin{align}
\nonumber
&\dive_{g}(X)\\
\nonumber
&=u^{-2\frac{p-1}{n-p}}g_{\R^n}^{ik}\left(\frac{\partial X_k}{\partial x_i}-G^l_{ik}X_l-\frac{p-1}{n-p}X_i\frac{\partial_{k} u}{u}-\frac{p-1}{n-p}X_k\frac{\partial_{i} u}{u}+g_{ik}^{\R^n}\frac{p-1}{n-p}\left\langle \frac{D u}{u}, X\right\rangle_{\R^n}\right)\\
&=u^{-2\frac{p-1}{n-p}}\dive_{g_{\R^n}} X+\frac{(n-2)(p-1)}{(n-p)}\left\langle \frac{D u}{u}, X \right\rangle_{g}.
\end{align}
where $\left\langle\cdot ,\cdot \right\rangle_{\R^n}$ and $\left\langle\cdot ,\cdot \right\rangle_{g}$ are the scalar products associated to $g_{\R^n}$ and $g$ respectively. Setting $X=|D u|^{p-2}D u$ and recalling that $\Delta_p u=0$ we obtain
\begin{align}\label{formPlap}
\dive_g(|D u|^{p-2}D u)&=\frac{(n-2)(p-1)}{n-p}\left\langle \frac{D u}{u}, |D u|^{p-2}D u\right\rangle_g\\
\nonumber
&=\frac{(n-2)(p-1)}{n-p}\frac{|D u|^{p-2}}{u}\left\langle D u, D u\right\rangle_g.
\end{align}
By standard computations it is easy to see that
%\[
%\frac{|Du|}{u^{\frac{n-1}{n-p}}}=\frac{n-p}{(n-2)(p-1)}|\nabla \psi|_g
%\]
\begin{align}\label{formgrad}
|D u|^{p-2}= u^{\frac{(p-1)(p-2)}{n-p}} |D u|_{g}^{p-2}.
\end{align}
Using \eqref{formgrad} in \eqref{formPlap} we get
\begin{align*}
\dive_g(u^{\frac{(p-1)(p-2)}{n-p}} |D u|_{g}^{p-2}D u)=\frac{(n-2)(p-1)}{n-p} u^{\frac{(p-1)(p-2)}{n-p}-1}|D u|_g^{p}
\end{align*}
we conclude that
\begin{align}\label{eq_conf}
\Delta_{p;g}u=\dive_g(|D u|_g^{p-2}D u)=(p-1)\frac{|D u|_g^p}{u}.
\end{align}
\begin{lemma}
Let $u\in C^{\infty}(\mathbb{R}^n\setminus \overline{\Omega})$ be a positive solution to $\Delta_p u=0$ in $\mathbb{R}^n\setminus \overline{\Omega}$. Then
\[
\Delta_{p;g}(\log u)=0\qquad \mbox{in}\ \mathbb{R}^n\setminus \overline{\Omega}.
\]
\end{lemma}
\begin{proof}
Let $f:=\log u$. Clearly, $\nabla f= \frac{\nabla u}{u}$ and $|\nabla f|_g^{p-2}=\frac{|\nabla u|_g^{p-2}}{u^{p-2}}$. Thus,
\begin{align*}\Delta_{p;g} f=\dive_g(|\nabla f|_{g}^{p-2}\nabla f)=\dive_g\left(\frac{|\nabla u|_g^{p-2}}{u^{p-2}}\frac{\nabla u}{u}\right).\end{align*}
Therefore,
	\begin{align*}\Delta_{p;g} f&=u^{1-p}\Delta_{p;g}u+\left\langle \nabla u^{1-p}, |\nabla u|_{g}^{p-2}\nabla u\right\rangle_g\\
	&=u^{1-p}\Delta_{p;g}u+(1-p)u^{-p}|\nabla u|^p_g
\end{align*}
and recalling \eqref{eq_conf} and using $|Du|_g=|\nabla u|_g$ we get the thesis.
\end{proof}
Keeping in mind formulas \eqref{cambioRic} and \eqref{cambiocon}, recalling that $\psi=-\frac{(n-2)(p-1)}{n-p}\log u=-\frac{(n-2)(p-1)}{n-p}f$ and $R^{\R^n}_{\alpha\beta}=0$ we obtain
\begin{align}\label{ric}
R_{\alpha\beta}^{g}=&-\frac{(n-2)(p-1)}{n-p}D_{\alpha}D_{\beta}f +\frac{(n-2)(p-1)^2}{(n-p)^2}\partial_{\alpha}f\partial_{\beta}f\\
\nonumber
&-\frac{p-1}{n-p}\Big(\Delta f+\frac{(n-2)(p-1)}{n-p}|Df|^2\Big)e^{-\frac{2(p-1)f}{n-p}}g_{\alpha\beta}
\end{align}
and
\begin{align}\label{cambiocon2}
D_{\alpha}D_{\beta}f=\nabla_{\alpha}\nabla_{\beta}f+\frac{p-1}{n-p}\Big(2\partial_{\alpha}f\partial_{\beta}f-|\nabla f|_{g}^2g_{\alpha\beta}\Big).
\end{align}
Using \eqref{cambiocon2} in \eqref{ric} we get
\begin{align}\label{cambiocon3}
R_{\alpha\beta}^{g}=&-\frac{(n-2)(p-1)}{n-p}\nabla_{\alpha}\nabla_{\beta}f-\frac{(n-2)(p-1)^2}{(n-p)^2}\partial_{\alpha}f\partial_{\beta}f+\\
\nonumber
&+\frac{(n-2)(p-1)^2}{(n-p)^2}|\nabla f|^2_g g_{\alpha\beta}\\
\nonumber
&-\frac{p-1}{n-p}\Delta f e^{-\frac{2(p-1)f}{n-p}}g_{\alpha\beta}-\frac{(n-2)(p-1)^2}{(n-p)^2}|Df|^2 e^{-\frac{2(p-1)f}{n-p}}g_{\alpha\beta}\\
\nonumber
%&=-\frac{(n-2)(p-1)}{n-p}\nabla_{\alpha}\nabla_{\beta}f-\frac{(n-2)(p-1)^2}{(n-p)^2}\partial_{\alpha}f\partial_{\beta}f+\frac{(n-2)(p-1)^2}{(n-p)^2}|\nabla f|^2_g g_{\alpha\beta}\\
%\nonumber
%&+\frac{p-1}{n-p} g_{\alpha\beta}\Delta_g f
&\!\!\!\!\!=\nabla_{\alpha}\nabla_{\beta}\psi-\frac{\partial_{\alpha}\psi\partial_{\beta}\psi}{n-2}+\frac{|\nabla \psi|^2_g}{n-2}g_{\alpha\beta}+\frac{1}{n-2}e^{-\frac{2(p-1)f}{n-p}}\left(\Delta\psi-|D\psi|^2\right)g_{\alpha\beta}\\
\nonumber
&\!\!\!\!\!=\nabla_{\alpha}\nabla_{\beta}\psi-\frac{\partial_{\alpha}\psi\partial_{\beta}\psi}{n-2}+\frac{|\nabla \psi|^2_g}{n-2}g_{\alpha\beta}+\frac{1}{n-2}\Delta_g\psi g_{\alpha\beta}
\end{align}
where in the last equality we used \eqref{cambioLap}.
Since \[0=\Delta_{g;p}\psi=|\nabla \psi|_{g}^{p-2}\Delta_g \psi+\left\langle \nabla|\nabla \psi|_{g}^{p-2},\nabla \psi\right\rangle_g\] we obtain
\begin{align}\label{espLap}
	\Delta_g \psi&=-(p-2)\frac{\nabla^2\psi(\nabla \psi,\nabla \psi)}{|\nabla \psi|_g^2}\\
	\nonumber
	&=-\frac{p-2}{2}\frac{\left\langle \nabla|\nabla \psi|_g^2, \nabla \psi\right\rangle_g}{|\nabla \psi|_g^2}.
\end{align}
Using \eqref{espLap} in \eqref{cambiocon3} we can write
%\begin{align}\label{cambiocon4}
%R_{\alpha\beta}^{g}=&-\frac{(n-2)(p-1)}{n-p}\nabla_{\alpha}\nabla_{\beta}f-\frac{(n-2)(p-1)^2}{(n-p)^2}\partial_{\alpha}f\partial_{\beta}f+\\
%\nonumber
%&+\left(\frac{(n-2)(p-1)^2}{(n-p)^2}|\nabla f|^2_g -\frac{(p-1)(p-2)}{n-p}\frac{\nabla^2 f(\nabla f,\nabla f)}{|\nabla f|_g^2} \right)g_{\alpha\beta},
%\end{align}
%setting $w:=\frac{(n-2)(p-1)}{(n-p)}f$ the previous formula can be written as
\begin{align}\label{cambiocon5}
R_{\alpha\beta}^{g}=\nabla_{\alpha}\nabla_{\beta}\psi-\frac{\partial_{\alpha}\psi\partial_{\beta}\psi}{n-2}+\left(\frac{|\nabla \psi|^2_g}{n-2} -\frac{p-2}{n-2}\frac{\nabla^2 \psi(\nabla \psi,\nabla \psi)}{|\nabla \psi|_g^2} \right)g_{\alpha\beta}
\end{align}
and in particular
\begin{align*}
\ric_g-\nabla^2 \psi+\frac{d\psi\otimes d\psi}{n-2}=\left(\frac{|\nabla \psi|_g^2}{n-2}-\frac{p-2}{n-2}\frac{\nabla^2 \psi(\nabla \psi,\nabla \psi)}{|\nabla \psi|_g^2}\right)g.
\end{align*}
We are now in position to reformulate problem \eqref{prob_ex} as 
\begin{align}\label{prob_ex_rif}
\left\{\begin{array}{lll}
\qquad\qquad\qquad\,\Delta_{p;g} \psi=0 & \mbox{in} & M\\
\ric_g-\nabla^2 \psi+\frac{d\psi\otimes d\psi}{n-2}=\left(\frac{|\nabla \psi|_g^2}{n-2}-\frac{p-2}{n-2}\frac{\nabla^2 \psi(\nabla \psi,\nabla \psi)}{|\nabla \psi|_g^2}\right)g & \mbox{in} & M\\
\qquad\qquad\qquad\qquad\psi=0 & \mbox{on} & \partial M\\
\qquad\qquad\qquad\,\,\,\,\,\psi(x)\to +\infty & \mbox{as} & |x|\to\infty.
\end{array}\right.
\end{align}
We explicitly observe that if $p=2$ then \eqref{prob_ex_rif} coincides with the problem studied in \cite{AgoMaz1}.

We conclude this part recalling the useful relation between $\abs{\nabla\psi}_g$ and $\abs{Du}$:

\begin{equation}\label{cambgrad}
|\nabla\psi|_g=\frac{(n-2)(p-1)}{n-p}\frac{|Du|}{u^{\frac{n-1}{n-p}}}
%& d\sigma_g=u^{\frac{(p-1)(n-1)}{(n-p)}} d\sigma.
\end{equation}

\subsection{The geometry of the level sets of $u$ and $\psi$}
%We fix on $\mathcal{U}_{\delta}:=\{s_0-\delta<\psi<s_0+\delta\}$. We know that the metric $g$ can be written as
%\[
%g=\frac{d\psi\otimes d\psi}{|\nabla \psi|_g^2}+g_{ij}\left(\psi, \theta^1,\ldots, \theta^{n-1}\right) d\theta^i\otimes d\theta^j
%\]
%where the latin indices vary between $1$ and $n-1$. A similar expression can be obtained for the Euclidean metric in terms of the local coordinates $\{u, \theta^1,\ldots, \theta^{n-1}\}$.
Let us consider the $g^{\mathbb{R}^n}-$unit vector field \[\nu:=-Du/|Du|=D\psi/|D\psi|\] and the $g-$unit vector field \[\nu_g:=-\nabla u/|\nabla u|_g=\nabla\psi/|\nabla \psi|_g.\] Accordingly, we consider the second fundamental forms $h$ and $h_g$ of the level sets of $u$ and $\psi$ with respect to the Euclidean metric $g^{\mathbb{R}^n}$ and the conformally-related ambient metric $g$ are respectively given by
\begin{equation*}
h_{ij}=-\frac{D^2_{ij}u}{|Du|}=\frac{D^2_{ij}\psi}{|D\psi|},\qquad h_{ij}^{g}=-\frac{\nabla^2_{ij}u}{|\nabla u|_g}=\frac{\nabla^2_{ij}\psi}{|\nabla\psi|_g}\quad \mbox{for}\ i,j=1,\ldots, n-1.
\end{equation*}
Taking the trace of the above expressions with respect to the induced metric we obtain the following expressions for the mean curvatures in the two ambients
\begin{equation}\label{meancurv}
H=-\frac{\Delta u}{|D u|}+\frac{D^2 u(D u,D u)}{|D  u|^3},\qquad H_g= \frac{\Delta_g \psi}{|\nabla \psi|_g}-\frac{\nabla^2\psi(\nabla \psi,\nabla \psi)}{|\nabla \psi|_g^3}.
\end{equation}
Recalling that $\Delta_p u=0$ and $\Delta_{g;p} \psi=0$ we have
\begin{equation}
\label{meancurvf}
H=\frac{p-1}{p} \frac{\left\langle D|D u|^p, D u\right\rangle}{|D u|^{p+1}}=(p-1)\frac{D^2 u(D u,D u)}{|D u|^3},
\end{equation}
and
\begin{equation}\label{meancurv2}
\qquad H_g=-\frac{p-1}{p} \frac{\left\langle \nabla|\nabla \psi|_g^p, \nabla \psi\right\rangle_p}{|\nabla \psi|_g^{p+1}}=-(p-1)\frac{\nabla^2\psi(\nabla\psi,\nabla\psi)}{|\nabla\psi|_g^3}.
\end{equation}
The second fundamental forms $h$ and $h_g$ are related by the following formula:
\begin{equation*}
h_g(X, Y)=u^{\frac{p-1}{n-p}}\Big(h(X, Y) - \frac{p-1}{n-p}\frac{\abs{Du}}{u}\langle X, Y\rangle\Big),
\end{equation*}
for any $X, Y$ tangent vectors to the level sets of $u$. Tracing the above identity with respect to $g$ we obtain the useful relation between the mean curvatures $H$ and $H_g$
\begin{equation}
\label{meancurv1}
H_g=u^{-\frac{p-1}{n-p}}\bigg(H-\frac{(n-1)(p-1)}{(n-p)}\frac{\abs{Du}}{u}\bigg).
\end{equation}
Finally, we recall the relation between the Lebesgue measure $d\mu$ and the volume measure $d \mu_g$ induced by $g$ on $M$ 
\begin{equation}
\label{meas1}
d \mu_g= u^{(p-1)\frac{n}{n-p}} d\mu
\end{equation}
and the relation between $(n-1)-$dimensional Hausdorff measure $\mathcal{H}^{n-1}=d\sigma$ and the surface element $d \sigma_g$ induced by $g$
\begin{equation}
\label{meas2}
d\sigma_g = u^{(p-1)\frac{n-1}{n-p}} d\sigma.
\end{equation}
\subsection{The conformal version of the main Theorems}
We start introducing the conformal version of the functions $V_q^p$ introduced in \eqref{defvqp}. Fix $p\in (1, n), q\in [0,\infty)$ and let $\psi$ be as in \eqref{defpsi}. We define $\Psi_q^p(t):[0, +\infty)\to [0, \infty)$ by
\begin{equation}
\label{psi}
\Psi_q^p(s)=\int_{\{\psi=s\}}\abs{\nabla\psi}_g^{q(p-1)}d\sigma_g.
\end{equation}

\begin{remark}
\label{facts}
Clearly,
\begin{equation*}
\Psi_0^p(s)=\left| \{\psi=s\}\right|_g=\sigma_g(\{\psi=s\})
\end{equation*}
Moreover, for $q=1$ it follows from $\Delta_{p,g}\psi=0$ and the Divergence Theorem that the function
\begin{equation}
\label{bound1}
\Psi_1^p(s)=\int_{\{\psi=s\}}|\nabla\psi|_g^{p-1} d\sigma_g
\end{equation}
is constant in $[0,\infty)$.
\end{remark} 
Remark \ref{facts} together with Proposition \ref{mono_and} readily implies that $\Psi_q^p$ is bounded.
\begin{lemma} 
\label{bound}
Let $p\in (1, n)$ and $q\in [1,\infty)$. Then there exists $C=C(\Omega,n,p)>0$ independent of $s$ such that 
\begin{equation*}
\Psi_q^p (s) \leq C
\end{equation*}
for any $s\in [0,\infty)$. 
\end{lemma}
\begin{proof}
By Theorem \ref{lewis}, Proposition \ref{mono_and} and formula \ref{cambgrad}, there exists a constant $C=C(\Omega,n,p)>0$ such that
\begin{equation*}
0<\abs{\nabla\psi}_g\leq C
\end{equation*}
in $M=\R^n\setminus\overline{\Omega}$. Then, we can write
\begin{equation*}
\Psi_q^p(s)=\int_{\{\psi=s\}}\abs{\nabla\psi}_g^{(p-1)(q-1)}\abs{\nabla\psi}_g^{p-1} d\sigma_g\leq C^{(p-1)(q-1)}\int_{\{\psi=s\}}\abs{\nabla\psi}^{p-1}_g d\sigma_g.
\end{equation*}
Since, as noticed in Remark \ref{facts}, the integral on the right hand side of the above inequality is constant, the claim is proved.
\end{proof}

We are now going to state the conformal version of Theorem \ref{mainth}. Here, as well as in the rest of the paper, we are following Notation \ref{notation}, with $g$ in place of $g^{\R^n}$ and $\psi$ in place of $u$.
\begin{theorem}\label{mainconf}
Let $M$, $g$, and $\psi$ be defined as in \eqref{m}, \eqref{defmet} and \eqref{defpsi}, and let  $0 \leq q < \infty$. Let $\Psi_q^p: [0, \infty) \to \R$ be the function defined in \eqref{psi}. Then, for every $(p, q) \in \Lambda$, the function $\Psi_p^q$ is differentiable and the following assertions hold true.
\begin{enumerate}
\item[(i)]
For any $s \geq 0$ the derivative of $\Psi_q^p$ satisfies the following relation.
\begin{equation}
\label{derpsi2}
\begin{split}
\big(\Psi_q^p\big)'(s) 
&=-(q-1) \int\limits_{\{\psi = s\}} \abs{\nabla \psi}_g^{q(p-1) - 1} H_g d\sigma_g \\ 
&=-{\big(q-1\big)}e^{\frac{n-p}{(n-2)(p-1)}s}\bigintsss\limits_{\{{\psi\geq s}\}}\abs{\nabla\psi}_g^{q(p-1)-3}\Bigg[\left\vert\nabla^2_T \psi - \frac{\Delta^g_T \psi}{n-1}\,g_T\right\vert_{g_T}^2 \\
&\quad\qquad\!\!\!\!\!\!\!\!+ \Big(q(p-1)-1\Big)\bigabs{\nabla_T\abs{\nabla\psi}_g}_{g_T}^2 \\
&\quad\qquad\!\!\!\!\!\!\!\! +\big(p-1\big)^2 \left[q - 1 - \frac{(n-p)}{(p-1)(n-1)}\right]\bigg\langle\nabla\abs{\nabla\psi}_g,\frac{\nabla\psi}{\abs{\nabla\psi}_g}\bigg\rangle_g^2\Bigg]e^{-\frac{n-p}{(n-2)(p-1)}\psi} d\mu_g.
\end{split}
\end{equation}
In particular, $\left(\Psi_q^p\right)'(s) \leq 0$ for any $s \in [0, \infty)$.
\item[(ii)]
If $\big(\Psi_p^q\big)'(s_0)=0$ for some $(p, q)\in\Lambda$ and $s_0\geq 0$, the manifold $(\{\psi\geq s_0\}, g)$ is isometric to $\big([s_0, + \infty) \times \{\psi=s_0\}, d\rho\otimes d\rho + g_{|\{\psi=s_0\}}\big)$, where $\rho$ is the $g$-distance to $\{\psi=s_0\}$ and $\psi$ is an affine function of $\rho$. Moreover $\left(\{\psi =s_0\}, g_{\mid\{\psi = s_0\}}\right)$ is a constant curvature sphere.
\end{enumerate}
\end{theorem}
It is easy to check that the above result implies our Theorem \ref{mainth}:
\begin{proof}[Proof of Theorem \ref{mainth} after Theorem \ref{mainconf}]
Let $u$ be a solution to problem \eqref{prob_ex}. Let us consider the metric $g$ defined in \eqref{defmet} and the function $\psi$ defined in \eqref{defpsi}, so that the system \eqref{prob_ex_rif}
is satisfied in $M=\mathbb{R}^n\setminus \overline{\Omega}$. 
Straightforward computations involving \eqref{cambgrad} and \eqref{meas2} show the following relation between the functions $V_q^p$ and $\Psi_q^p$ 
\begin{equation}
\label{confV}
\begin{split}
V_q^p(t)&=\bigg(\frac{(n-p)}{(n-2)(p-1)}\bigg)^{q(p-1)}{C_p(\Omega)}^{\frac{(n-1)(q-1)}{(n-p)}}\Psi_q^p\Big(-\frac{(n-2)(p-1)}{(n-p)}\log t\Big) \\
-t\frac{d V_q^p}{dt} (t)&=\bigg(\frac{(n-p)}{(n-2)(p-1)}\bigg)^{\!\!\!q(p-1)-1}{\!\!C_p(\Omega)}^{\frac{(n-1)(q-1)}{(n-p)}}\frac{d\Psi_q^p}{dt}\Big(\!\!-\frac{(n-2)(p-1)}{(n-p)}\log t\Big),
\end{split}
\end{equation}
that, combined with the first identity in \eqref{derpsi2} and \eqref{meancurv1}, implies the first expression for the derivative of $V_q^p$ in \eqref{vqpder}. Getting the non-negative expression for the derivative of $V_q^p$ from \eqref{derpsi2} is just a matter of lengthy computations carried out by means of the various identities shown in Subsections 3.1 and 3.2. 
More precisely, one has to check that the quadratic quantities appearing in \eqref{derpsi2} are related to those of \eqref{vqpder} as follows.
\begin{equation}
\label{quadrato1}
\begin{split}
{\left\vert\nabla^2_T \psi - \frac{\Delta^g_T \psi}{n-1}\,g_T\right\vert_{g_T}^2} &= e^{\frac{4\psi}{n-2}}\left\vert D^2_T \psi - \frac{\Delta_T \psi}{n-1}\,g^{\R^n}_T\right\vert_{g^{\R^n}_T}^2 \\ 
&= \left[\frac{(p-1)(n-2)}{(n-p)}\right]^2 u^{-2\frac{n + p -2}{n-p}}\left\vert D^2_T u - \frac{\Delta_T u}{n-1}\,g^{\R^n}_T\right\vert_{g^{\R^n}_T}^2,
\end{split}
\end{equation}

\begin{equation}
\label{quadrato2}
\left\vert \nabla_T \abs{\nabla \psi}_{g}\right\vert_{g_T}^2 = e^{\frac{4\psi}{n-2}} \left\vert D_T \abs{D \psi}_{g^{\R^n}}\right\vert_{g^{\R^n}_T}^2 = \left[\frac{(p-1)(n-2)}{(n-p)}\right]^2 u^{-2\frac{n + p -2}{n-p}} \left\vert D_T \abs{D u}_{g^{\R^n}}\right\vert_{g^{\R^n}_T}^2,
\end{equation}
and
\begin{equation}
\label{quadrato3}
\begin{split}
\left\langle\nabla\abs{\nabla\psi}_g,\frac{\nabla\psi}{\abs{\nabla\psi}_g}\right\rangle_g^2 &= \left[\frac{H_g \abs{\nabla\psi}_g}{p-1}\right]^2 \\ 
&= \left[\frac{n-2}{n-p}\right]^2  u^{-2\frac{n + p -2}{n-p}}\abs{Du}^2\left[H- \frac{(n-1)(p-1)}{(n-p)} \frac{\abs{Du}}{u}\right]^2.
\end{split}
\end{equation}
 
We now turn to prove the rigidity statement of Theorem \ref{mainth}. Assume that $(V_q^p)'(t_0)=0$ for some $t_0\in (0, 1]$ then \eqref{confV} implies that $(\Psi_q^p)'(s_0)=0$ for the corresponding $s_0\in [0, +\infty)$ and then the rigidity part of Theorem \ref{mainconf} implies that $\{\psi\geq s_0\}$ is isometric to one half round cylinder with totally geodesic boundary. In particular, $\left(\{u = t_0\}, g_{\mid \{u =t_0\}}\right)$ is a constant curvature sphere. Since the conformal factor relating $g$ and $g_{\R^n}$ is a function of $u$, these two metrics coincide up to a multiplicative constant on the level sets of $u$, and thus, up to a translation $\{u=t_0\}=\partial B(0, R_0)$ for some $R_0>0$. Let now $v:\R^n\setminus B(0, R_0) \mapsto \R$ be defined by
\[
v(x)=t_0\bigg(\frac{R_0}{\abs{x}}\bigg)^{\frac{n-p}{p-1}}. 
\]
Then it is easy to see that $v$ solves  
\begin{align}\
\left\{\begin{array}{lll}
\Delta_p v=0 & \mbox{in} & \R^n\setminus \overline{B}_{R_0}\\
\quad \,\,v=t_0 & \mbox{on} & \partial B_{R_0}\\
\,v(x)\to 0 & \mbox{as} & |x|\to\infty.
\end{array}\right.
\end{align}
Since $u$ solves the same problem then by the uniqueness of solutions given in Theorem \ref{lewis} we get $u=v$ in $\R^n\setminus B_{R_0}$, and in turn $u$ is rotationally symmetric in this region. Finally, since $v$ can be extended to $\R^n\setminus\{0\}$, we have that $u$ and $v$ are both analytic function (recall Remark \ref{elliptic}) coinciding on an open subset of $M$, and thus they must coincide on the whole $\R^n\setminus\Omega$.
\end{proof}

Define now the function $\Psi_\infty^p: [0, + \infty) \to \R$ as
\[
\Psi_\infty^p (s) = \sup_{\{\psi=s\}} \abs{\nabla \psi}_g.
\]
The conformal version of \ref{mainthsup} reads as follows.

Again by \eqref{cambgrad} and \eqref{meancurv1}, and arguing as above for what it concerns the rigidity part, the conformal version of Theorem \ref{mainthsup} is easily seen to be the following.

\begin{theorem}
\label{sharpbound}
Let $(M, g, \psi)$ be a solution to \eqref{prob_ex_rif}. Then, the following assertions hold true.
\begin{enumerate}
\item[i)]
The function $\Psi_\infty^p$ is monotone non-increasing.
Moreover, $\Psi_\infty^p(s)=\Psi_\infty^p(s_1)$ for some $s_1 >  s$ if and only if $\{\psi \geq s\}$ is isometric to one half round cylinder with totally geodesic boundary. 
\item[ii)]

Let $x_s \in \{\psi =s\}$ be a maximum point of $\abs{\nabla \psi}_g$ on $\{\psi=s\}$. Then
\begin{equation}
\label{hgineq}
H_g(x_s) \geq 0,
\end{equation}
with equality occurring if and only if $\{\psi\geq s\}$ is isometric to one half round cylinder with totally geodesic boundary.
\end{enumerate}

\end{theorem}

Indeed, the equivalence between $\Psi_\infty^p$ and $V_\infty^p$ is clear from \eqref{cambgrad}, while the inequality \eqref{hgineq} is equivalent to \eqref{supvder} by \eqref{meancurv1}. Moreover,the rigidity statements in the above theorem imply those of Theorem \ref{mainthsup} as already proved above. 
\newline{}

In light of the above discussion, we devote the rest of the paper to prove Theorem \ref{mainconf} and Theorem \ref{sharpbound}.
\section{Proof of Theorem \ref{mainconf}}
We start by explicitly computing the derivative of $\Psi_q^p$. Since we are always going to deal only with the conformal setting introduced in the previous section, we omit the subscript $g$ in all the subsequent results.
\begin{proposition}
\label{deri}
Let $\Psi_q^p: [0,\infty)\to\mathbb{R}$ be defined as in \eqref{psi}. Then $\Psi_q^p$ is a differentiable function and its derivative satisfies, for all $s\geq 0$,  
\begin{equation}
\label{derif}
\begin{split}
\big(\Psi_q^p\big)'(s)
&=\int_{\{\psi=s\}}\left\langle \nabla\abs{\nabla\psi}^{(p-1)(q-1) }\abs{\nabla\psi}^{p-2}, \frac{\nabla \psi}{|\nabla\psi|}\right\rangle d\sigma \\
&=-(q-1)\int_{\{\psi=s\}}{\abs{\nabla\psi}}^{q(p-1)-1}Hd\sigma.
\end{split}
\end{equation}

\end{proposition}
\begin{proof}
Fix $s_0\in [0,\infty)$ and let $s_0 < s$. We write
\[
\begin{split}
\Psi_q^p(s) - \Psi_q^p(s_0)&=\int_{\{\psi=s\}}\left\langle\nabla\psi\abs{\nabla\psi}^{q(p-1)-1}, \frac{\nabla\psi}{\abs{\nabla\psi}}\right\rangle d\sigma \\ 
&- \int_{\{\psi=s_0\}}\left\langle\nabla\psi\abs{\nabla\psi}^{q(p-1) - 1}, \frac{\nabla\psi}{\abs{\nabla\psi}}\right\rangle d\sigma.
\end{split}
\]
Using the Divergence Theorem, we can write the above quantity as
\begin{equation}
\label{passaggio}
\Psi_q^p(s)-\Psi_q^p(s_0)=\int_{\{s_0\leq \psi\leq s\}}\dive\big(\nabla\psi\abs{\nabla\psi}^{p-2}\abs{\nabla\psi}^{(p-1)(q-1)}\big)d\mu.
\end{equation}
Since $\Delta_p \psi=0$, we get
\[
\dive(\nabla\psi\abs{\nabla\psi}^{p-2}\abs{\nabla\psi}^{(p-1)(q -1)})=\la\nabla\abs{\nabla\psi}^{(p-1)(q-1)}, \nabla\psi\ra\abs{\nabla\psi}^{p-2},
\]
by the coarea formula we obtain
\[
\Psi_q^p(s) - \Psi_q^p(s_0)=\int_{s_0}^s\int_{\{\psi=\tau\}}\big\langle\big(\nabla\abs{\nabla\psi}^{(p-1)(q-1)}\big) \abs{\nabla\psi}^{p-2}, \frac{\nabla\psi}{\abs{\nabla\psi}}\big\rangle d\sigma d\tau.
\]
The first equality in \eqref{derif} then follows by the 
Fundamental Theorem of Calculus provided that
the function $I$ mapping 
\begin{equation}
\label{fun}
 \tau \mapsto \int_{\{\psi=\tau\}}\big\langle\big(\nabla\abs{\nabla\psi}^{(p-1)(q-1)}\big) \abs{\nabla\psi}^{p-2}, \frac{\nabla\psi}{\abs{\nabla\psi}}\big\rangle d\sigma
\end{equation}
is continuous.
In fact, fixed $\tau_0 \geq 0$ we have for any $\tau > \tau_0$
\[
\begin{split}
\bigabs{ I(\tau) - I(\tau_0)}  
&\leq \int_{\{\tau_0 \leq \psi \leq \tau\}} \bigabs{\dive \left((\nabla\abs{\nabla\psi}^{(p-1)(q-1)}\big) \abs{\nabla\psi}^{p-2}\right)}\ d\mu  \\ 
&\leq C\mu({\{\tau_0 \leq \psi \leq \tau\})}
\end{split},
\]
with $C=\sup_{\{\tau_0 \leq \psi \leq \tau\}}\nabla\abs{\nabla\psi}^{(p-1)(q-1)} \abs{\nabla\psi}^{p-2}$. 
The continuity of the function defined in \eqref{fun} then follows from the continuity of $\psi$ and the fact that, thanks to \eqref{meas1}, $\mu$ is absolutely continuous with respect to the Lebesgue measure.
 
The second equality in \eqref{derif} follows by  \eqref{meancurv2}.
\end{proof}

Applying the generalized Bochner formula provided in \cite[Proposition 3.3]{Valt} to a solution of \eqref{prob_ex_rif} leaves with the following identity.

%INIZIO COMMENTO

\begin{comment}
\begin{theorem}\label{Valto}
Let $M$ be an $n-$dimensional Riemannian manifold. Given $x\in M$, a domain $U$ containing $x$ and $v\in C^2(U)$, if $\nabla v\neq 0$ on $U$ we denote by
\[
A_v:=\frac{\nabla^2 v(\nabla v,\nabla v)}{|\nabla v|_g^2}.
\] 
Then,
\begin{align}
&\frac{1}{p}\left[|\nabla v|_g^{p-2} \Delta_g |\nabla v|_g^p + (p-2)|\nabla v|_g^{p-4}\nabla^2|\nabla v|_g^p(\nabla v, \nabla v)\right]=\\
\nonumber
&|\nabla v|_g^{2p-4}\left\{|\nabla v|_g^{2-p}\left[\left\langle \nabla\Delta_p v,\nabla v\right\rangle-(p-2)A_v\Delta_{g;p} v\right]+|\nabla^2 v|_g^2 +p(p-2)|A_v^2+Ric(\nabla v,\nabla v)\right\}.
\end{align}
We also have
\begin{align}
|\nabla v|_g^{2p-4}\left(|\nabla^2\psi|^2+p(p-2)A_v\right)\geq \frac{(\Delta_p v)^2}{n^{'}}+\frac{n^{'}}{n^{'}-1}\left(\frac{\Delta_p v}{n^{'}}-(p-1)|\nabla v|_g^{p-2}A_v\right)^2,
\end{align}
for all $n^{'}\geq n$.
\end{theorem}

Applying Theorem \ref{Valto} to a solution to \eqref{prob_ex_rif} we immediately get the following
\end{comment}

%FINE COMMENTO

\begin{lemma}
\label{bochp}
Let $\psi$ be a solution to \eqref{prob_ex_rif}. Then,
%\begin{align}\label{pBoch}
%&\Delta|\nabla \psi|^p+ (p-2) \frac{\nabla^2 |\nabla \psi|^p(\nabla \psi,\nabla \psi)}{|\nabla \psi|^2}\\
%nonumber
%&=p|\nabla \psi|^{p-2}\left(|\nabla^2 \psi|^2+p(p-2)\Big|\frac{\nabla^2 \psi(\nabla \psi,\nabla \psi)}{|\nabla \psi|^2}\Big|^2\right)+\frac{p(n-p)}{2(n-2)}|\nabla \psi|^{p-2}\left\langle \nabla|\nabla \psi|^2, \nabla \psi\right\rangle,
%\end{align}
%or equivalently,
\begin{align}\label{pBoch2}
\Delta|\nabla \psi|^p &+ (p-2) \frac{\nabla^2 |\nabla \psi|^p(\nabla \psi,\nabla \psi)}{|\nabla \psi|^2}-\frac{n-p}{n-2}\left\langle \nabla|\nabla \psi|^p,\nabla \psi\right\rangle=\\
\nonumber
&=p|\nabla \psi|^{p-2}\left(|\nabla^2 \psi|^2+{p(p-2)}\bigg\langle \nabla|\nabla \psi|,\frac{\nabla \psi}{\abs{\nabla\psi}}\bigg\rangle^2\right).
\end{align}
%Moreover, for every $n^{'}\geq n$, 
%\begin{align}\label{menchia}
%|\nabla^2\psi|^2\geq c_{n^{'},p} A_\psi^2,
%\end{align}
%where
%\[
%c_{n^{'},p}:=\frac{n^{'}}{n^{'}-1}(p-1)^2-p(p-2).
%\]
\end{lemma}

We now use \eqref{pBoch2} to relate the divergence of the vector field
\[
X:=\nabla\abs{\nabla\psi}^{(p-1)(q-1)}\abs{\nabla\psi}^{p-2}
\]
to the derivative of $\Psi_q^p$.

\begin{proposition}
\label{bochgen}
Let $\psi$ be a solution to \eqref{prob_ex_rif}. The following identity holds:
\begin{equation}
\label{bochgenf}
\begin{split}
\dive\big(X\big)- \frac{n-p}{n-2}\langle X, \nabla\psi\rangle=&\,\,\Big((p-1)(q-1)\Big)\abs{\nabla\psi}^{(p-1)q-3} \\     
&\!\!\!\!\times \Bigg[\abs{\nabla^2\psi}^2 + \Big(q(p-1)-p-1\Big)\bigabs{\nabla\abs{\nabla\psi}}^2 \\
&\!\!\!\!+\big(p-2\big)\bigg(\Big(\frac{\langle\nabla\abs{\nabla\psi}, \nabla\psi\rangle}{\abs{\nabla\psi}}\Big)^2-\frac{\nabla^2\abs{\nabla\psi}(\nabla\psi, \nabla\psi)}{\abs{\nabla\psi}}\bigg)\Bigg].
\end{split}
\end{equation}
\end{proposition}

\begin{proof}
Simple computations give:
\begin{align}\label{formula1}
\nabla\abs{\nabla\psi}^p&=p\abs{\nabla\psi}^{p-1}\nabla\abs{\nabla\psi},\\
\label{formula2}
\nabla^2\abs{\nabla\psi}^p(\nabla\psi,\nabla\psi)&=p(p-1)\abs{\nabla\psi}^{p-2}\langle\nabla\abs{\nabla\psi},\nabla\psi\rangle^2 + p\abs{\nabla\psi}^{p-1}\nabla^2\abs{\nabla\psi}(\nabla\psi, \nabla\psi).
\end{align}
and 
\begin{equation}
\label{id1}
X=\frac{(p-1)(q-1)}{p}\abs{\nabla\psi}^{q(p-1)-p-1}\nabla\abs{\nabla\psi}^p.
\end{equation}
Using \eqref{formula1} and \eqref{id1} we have
\begin{equation}
\label{contazzo1}
\begin{split}
\dive(X)&=\frac{(p-1)(q-1)}{p}\abs{\nabla\psi}^{q(p-1)-p-1}\Delta\abs{\nabla\psi}^p \\
&\quad+\big((p-1)(q-1)\big)\big(q(p-1)-p-1\big)\abs{\nabla\psi}^{q(p-1)-3}\bigabs{\nabla\abs{\nabla\psi}}^2.
\end{split}
\end{equation}
Plugging \eqref{pBoch2} into the above relation, and using identity \eqref{id1}, we obtain
\begin{equation}
\label{contazzo2}
\begin{split}
\dive(X)-\frac{n-p}{n-2}\langle X, \nabla\psi\rangle&=\Big((p-1)(q-1)\Big)\abs{\nabla\psi}^{q(p-1)-3}\\
&\,\,\,\,\times\bigg[\abs{\nabla^2\psi}^2+\Big(q(p-1)-p-1\Big)\bigabs{\nabla\abs{\nabla\psi}}^2 \\
&\quad\,\,-\frac{(p-2)}{p}\frac{\nabla^2\abs{\nabla\psi}^p(\nabla\psi,\nabla\psi)}{\abs{\nabla\psi}^p} \\ 
&\quad\,\,+ {p(p-2)}{\left\langle\nabla\abs{\nabla\psi},\frac{\nabla\psi}{\abs{\nabla \psi}}\right\rangle}^2\bigg].
\end{split}
\end{equation}
The conclusion follows using \eqref{formula2}.
%\[
%\nabla^2\abs{\nabla\psi}^p(\nabla\psi,\nabla\psi)=p(p-1)\abs{\nabla\psi}^{p-2}\langle\nabla\abs{\nabla\psi},\nabla\psi\rangle^2 + p\abs{\nabla\psi}^{p-1}\nabla^2\abs{\nabla\psi}(\nabla\psi, \nabla\psi)
%\]
%we obtain the claim.
\end{proof}
We are now going to prove a refinement for $p$-harmonic functions of a relation appeared for the first time in \cite{sternberg} and proved in a Riemannian setting e.g. in \cite[Proposition 18]{farina}. It will play a key  role in proving that $(\Psi_q^p)'$ has a sign and in characterizing the associated rigidity. The tangential elements appearing below are to be understood as in Notation \ref{notation}, with obvious changes.

\begin{proposition}
\label{ineq}
Let $(M, g)$ be a complete non-compact Riemannian manifold, and let $f\in C^2(M)$ be $p$-harmonic, for $p>1$.
\begin{enumerate}
\item[(i)]At any point $x\in M$ such that $\abs{\nabla f}(x)>0$ there holds
\begin{equation}
\label{ineqf}
\begin{split}
\abs{\nabla^2 f}^2 - \bigg(1 + \frac{(p-1)^2}{n-1}\bigg) \bigabs{\nabla\abs{\nabla f}}^2  = &\,\,\,\,\bigabs{\nabla^2_T f - \frac{\Delta_T f}{n-1}\,g_T}^2 \\  
& \!\!+\bigg(1 - \frac{(p-1)^2}{n-1}\bigg) \bigabs{\nabla_T\abs{\nabla f}}^2, 
\end{split}
\end{equation}

\item[(ii)] If for some $\rho_0 \in \R$  there hold $\abs{\nabla f}(x)> 0$  and
\begin{align}
\label{condrig1}
\bigabs{\nabla^2_T f - \frac{\Delta_T f}{n-1}\,g_T}^2(x)  = 0 \\ 
\label{condrig2}
\bigabs{\nabla_T\abs{\nabla f}}^2(x) =0
\end{align}
for any $x$ in $\{f  \geq \rho_0\}$,
then the Riemannian manifold $(\{f \geq \rho_0\}, g)$ is isometric to the warped product $([\rho_0, + \infty) \times \{f = \rho_0\}, d\rho \otimes d\rho + \eta^2 (\rho)\, g_{\mid \{f = \rho_0\}})$, where $\eta$ and $f$ are related as 
\begin{equation}
\label{eta}
f(x) = \rho_0 + c \int_{\rho_0}^{\rho(x)} \frac{d \tau}{\eta^{\frac{n-1}{p-1}}},
\end{equation}
for some constant $c > 0$.
\end{enumerate}
%In particular, the following inequality holds
%\begin{equation}
%\label{idtr}
%\abs{\nabla^2 f}^2 \geq 2\bigabs{\nabla_T\abs{\nabla f}}^2 + \bigg(1 + \frac{(p-1)^2}{n-1}\bigg) \bigg\langle \nabla \abs{\nabla f}, \frac{\nabla f}{\abs{\nabla f}} \bigg\rangle^2.
%\end{equation}
\end{proposition}

\begin{proof}
We consider on $x$ an orthonormal frame $\{e_1, \dots, \, e_{n-1}, \, e_n = \nabla f/ \abs{\nabla f}\}$. We can write the norm of $\nabla^2 f$ as follows
\begin{equation}
\label{scomphess}
\abs{\nabla^2 f}^2=\abs{\nabla^2_T f}^2 + 2\sum_{j}^{n-1}\abs{\nabla^2 f(e_n, e_j)}^2 + \abs{\nabla^2 f (e_n, e_n)}^2.
\end{equation}
Since, for a generic $2$-tensor $A$ on an $m$-dimensional manifold with a Riemannian metric $h$, the following identity relating norm and trace holds pointwise
\[
\abs{A}^2 = \frac{(\text{Tr} A)^2}{m} + \bigabs{A - \frac{\text{Tr} A}{m}h}^2,
\]
we can write the first term in the right hand side of \eqref{scomphess} as follows
\begin{equation}
\label{scomphesst}
\abs{\nabla^2_T f}^2=\frac{(\Delta_T f)^2}{n-1} + \left\vert\nabla^2_T f - \frac{\Delta_T f}{n-1}\,g\right\vert^2,
\end{equation}
We now exploit the $p$-harmonicity of $f$. Indeed, by
\[
\Delta_p f=\abs{\nabla f}^{p-2}\bigg(\Delta f + (p-2)\nabla^2 f\bigg(\frac{\nabla f}{\abs{\nabla f}}, \frac{\nabla f}{\abs{\nabla f}}\bigg) \bigg)=0
\]
and $\nabla f(x)\neq 0$, we have
\[
\Delta f = - (p-2)\nabla^2 f\bigg(\frac{\nabla f}{\abs{\nabla f}}, \frac{\nabla f}{\abs{\nabla f}}\bigg).
\]
The above identity implies 
\begin{equation}
\label{scomdeltat}
\Delta_T f=\Delta f - \nabla^2 f (e_n, e_n)=-(p-1) \nabla^2 f\bigg(\frac{\nabla f}{\abs{\nabla f}}, \frac{\nabla f}{\abs{\nabla f}}\bigg),
\end{equation}
that, plugged into \eqref{scomphesst}, gives
\begin{equation}
\label{scomphesst1}
\abs{\nabla^2_T f}^2=\frac{(p-1)^2}{n-1}\left\vert\nabla^2 f\bigg(\frac{\nabla f}{\abs{\nabla f}}, \frac{\nabla f}{\abs{\nabla f}}\bigg)\right\vert^2  + \left\vert\nabla^2_T f - \frac{\Delta_T f}{n-1}\,g\right\vert^2.
\end{equation}

We now turn our attention to the second and the third term in \eqref{scomphess}.
An easy computation shows
\begin{equation}
\label{idhess}
\nabla^2 f(e_n, e_j)= \langle \nabla \abs{\nabla f}, e_j\rangle
\end{equation}
for any $j=1, \dots, n-1$, and thus we have
\begin{equation}
\label{hesst}
\sum_{j}^{n-1}\abs{\nabla^2 f(e_n, e_j)}^2 =  \bigabs{\nabla_T \abs{\nabla f}}^2
\end{equation}
and
\begin{equation}
\label{hessen}
\bigabs{\nabla^2 f (e_n, e_n)}^2=\bigg\langle \nabla \abs{\nabla f}, \frac{\nabla f}{\abs{\nabla f}}\bigg\rangle^2.
\end{equation}
Finally, plugging \eqref{scomphesst1}, \eqref{hesst} and \eqref{hessen} into \eqref{scomphess} we obtain \eqref{ineqf}. %Inequality \eqref{idtr} follows using
%\[
%\bigabs{\nabla\abs{\nabla f}}^2 = \bigabs{\nabla_T\abs{\nabla f}}^2 +  \bigg\langle \nabla \abs{\nabla f}, \frac{\nabla f}{\abs{\nabla f}}\bigg\rangle^2 . 
%\]
\newline{}

Let us now assume $\abs{\nabla f} > 0$ on $\{ f \geq \rho_0 \}$ and conditions \eqref{condrig1}-\eqref{condrig2} hold on this set. Then, by the first assumption, we deduce by  standard results in differential geometry (see e.g. \cite[Theorem 2.2]{hirsh}) that 
$\{ f \geq \rho_0\}$ is diffeomorphic to $[\rho_0, + \infty) \times \{f = \rho_0\}$, and in particular there exist new coordinates $\{f, x^1, \dots,  x^{n-1}\}$ on $\{ f \geq \rho_0\}$ such that
\[
g = \frac{df \otimes df}{\abs{\nabla f}^2} + g_{ij}(f, x) dx^i \otimes dx^j,
\]
where, $i, j$ range in $1, \dots, n-1$ and $\{x^i\}$ are coordinates on $\{f = \rho_0\}$. Observe now that by \eqref{condrig1}, the function $\abs{\nabla f}$ is constant on each level set of $f$. In other words, it is a function of $f$ alone. We can then define a new coordinate $\rho = df/\abs{\nabla f}$  so that the metric becomes
\[
g = d\rho \otimes d\rho + g_{ij}(\rho, x) dx^i \otimes dx^j,
\]
with some abuse of notation. In this coordinates, standard computations show that the Hessian is computed as 
\begin{equation}
\label{hess1}
\nabla^2 f = f'' d\rho\otimes d\rho + f' \nabla^2 \rho = f'' d\rho \otimes d\rho + \frac{1}{2} f'\partial_\rho g_{ij}dx^i \otimes dx^j,
\end{equation}
where by $f'$ and $f''$ we denote the derivatives of $f$ with respect to $\rho$. 

Let us now consider, for any fixed point $x \in \{\rho \geq \rho_0\}$ the  orthonormal frame $\{e_1, \dots, \, e_{n-1}, \, e_n = \nabla f/ \abs{\nabla f} = \nabla \rho\}$, already used in the first part of this proof. Then we have
\[
\nabla^2 f (e_n, e_j) = \langle \nabla \abs{\nabla f}, e_j \rangle = 0
\]
by \eqref{condrig1}, and
\[
\nabla^2_T f = - \frac{p-1}{n-1} \nabla^2 f \left(\frac{\nabla f}{\abs{\nabla f}}, \frac{\nabla f}{\abs{\nabla f}}\right)
\]
by \eqref{condrig2} combined with \eqref{scomdeltat}.
In particular, the Hessian of $f$ can also be computed as 
\[
\nabla^2 f = \nabla^2 f \left(\frac{\nabla f}{\abs{\nabla f}}, \frac{\nabla f}{\abs{\nabla f}}\right) d\rho \otimes d\rho -\frac{p-1}{n-1} \nabla^2 f \left(\frac{\nabla f}{\abs{\nabla f}}, \frac{\nabla f}{\abs{\nabla f}}\right) g_{ij} dx^i\otimes dx^j.
\]
A comparison with \eqref{hess1} then gives the system of ordinary differential equations
\begin{equation}
\label{ode}
\partial_\rho \log{g_{ij}(\rho, x)} = -2 \frac{p-1}{n-1} \partial_\rho \log {f'(\rho)},
\end{equation}
that, integrated, yields
\[
g_{ij} (\rho, x) = g_{ij} (\rho_0, x) \left( \frac{f'(\rho_0)}{f'(\rho)} \right)^{2\frac{p-1}{n-1}}, 
\] 
that is, $g$ has the warped product structure claimed, with 
\[
\eta (\rho) = \left( \frac{f'(\rho_0)}{f'(\rho)} \right)^{\frac{p-1}{n-1}}.
\]
Expression \eqref{eta} is clearly equivalent to the above one.
\end{proof}
\begin{remark}
We point out that similar identities have been used to get numerous rigidity results for the equation $\Delta_p u= f(u)$, $f\in C^1(\mathbb{R})$ in $\mathbb{R}^n$. See \cite{Val1,Val2,Val3}.
\end{remark}

As a corollary of the above Proposition, we record the following refined Kato's inequalities for $p$-harmonic functions, together with a characterization of the equality case. We will not need this corollary in the sequel, but it actually is of some independent interest.

\begin{corollary}[Refined Kato's inequalities for $p$-harmonic functions]
Let $(M, g)$ be a complete non-compact Riemannian manifold, and let $f$ be a $p$-harmonic function, with $p \in (1, n)$.
\begin{enumerate}
\item[(i)] If $(p-1)^2 \leq n-1$, then , on any $x\in M$ such that $\abs{\nabla f}(x) > 0$,
\begin{equation}
\label{kato1}
\abs{\nabla^2 f}^2 \geq \left(1 + \frac{(p-1)^2}{n-1}\right) \bigabs{\nabla\abs{\nabla f}}^2.
\end{equation}
Moreover, if equality is achieved on $\{f \geq f_0\}$ for some $f_0 \in \R$, and $\abs{\nabla f} > 0$ in this region, then the same conclusion of $(ii)$ in Proposition \ref{ineq} holds.
\item[(ii)] If $(p-1)^2 > n-1$, then, on any $x \in M$ such that $\abs{\nabla f} (x) > 0$,
\begin{equation}
\label{kato2}
\abs{\nabla^2 f}^2 \geq 2 \bigabs{\nabla\abs{\nabla f}}^2.
\end{equation}
Moreover, if equality is achieved on $\{f \geq \rho_0\}$ for some $f_0 \in \R$, and $\abs{\nabla f} > 0$ in this region, then $(\{f \geq \rho_0\}, g)$ splits as a Riemannian product \\
\noindent
 $([\rho_0, +\infty) \times g_{\mid\{ f = \rho_0\}}, d\rho \otimes d\rho + g_{\mid \{f = \rho_0\}})$ and $f$ is an affine function of $\rho$.
\end{enumerate}

\end{corollary}
\begin{proof}
The assertions in (i) follow straightforwardly from Proposition \ref{ineq}.

Let now $(p-1)^2 > n-1$. Plugging  
\[
\bigabs{\nabla_T\abs{\nabla f}}^2 = \bigabs{\nabla\abs{\nabla f}}^2 -  \bigg\langle \nabla \abs{\nabla f}, \frac{\nabla f}{\abs{\nabla f}}\bigg\rangle^2, 
\]
into \eqref{ineqf} we obtain \eqref{kato2}.
Assume now equality holds in \eqref{kato2}. Then, by \eqref{ineqf} and the above identity we obtain
\[
\bigabs{\nabla^2_T f - \frac{\Delta_T f}{n-1}\,g_T}^2 + \left(\frac{(p-1)^2}{n-1} - 1\right)\left\langle  \nabla \abs{\nabla f}, \frac{\nabla f}{\abs{\nabla f}}\right\rangle^2 = 0,
\]
that in turn imply 
\begin{align}
\bigabs{\nabla^2_T f - \frac{\Delta_T f}{n-1}\,g_T}^2  =\, 0 \\
\left\langle  \nabla \abs{\nabla f}, \frac{\nabla f}{\abs{\nabla f}}\right\rangle^2 =  0.
\end{align}
In particular, by  \eqref{scomdeltat} and \eqref{hessen}, we deduce that $\abs{\nabla^2 f} = 0$. The isometry claimed  follows by \cite[Theorem 4.1 (i)]{AgoMaz1}. An easy consequence is that $f$ is an affine function of $\rho$, as it can also be deduced by following the proof of the aforementioned result.
\end{proof}

\begin{remark} 
Setting $p=2$ in identity \eqref{ineqf} gives the well-known refined Kato's inequality for harmonic functions (see e.g. \cite{SSY})
\[
\abs{\nabla^2 f}^2 \geq \frac{n}{n-1}\bigabs{\nabla\abs{\nabla f}}^2.
\]
A proof of the rigidity associated to the equality case can be found in \cite[Proposition 5.1]{bour}. For the case $p \neq 2$ our characterization of the equality cases seems to be new. However, refined Kato's inequalities for $p$-harmonic functions are provided also in \cite{cinesi}.
\end{remark}
We are ready to prove the following identity involving the derivative of $\Psi_q^p$, in fact the key step  to obtain Theorem \ref{mainconf}. 
\begin{theorem}
\label{idint}
Let $1<p<n$, and $q\geq 1$. Let $\Psi_q^p$ be defined as in \eqref{psi}. Then, for $0<s<S$, we have
\begin{equation}
\label{idintf1}
\begin{split}
&\Big(e^{-\frac{n-p}{(n-2)(p-1)}S}\big(\Psi_q^p\big)'(S)- e^{-\frac{n-p}{(n-2)(p-1)}s}\big(\Psi_q^p\big)'(s)\Big)= \\
&{\big(q-1\big)}\bigintsss\limits_{\{{s\leq\psi\leq S}\}}\abs{\nabla\psi}^{(p-1)q-3}\Bigg[\left\vert\nabla^2_T \psi - \frac{\Delta_T \psi}{n-1}\,g_T\right\vert^2 + \Big(q(p-1)-1\Big)\bigabs{\nabla_T\abs{\nabla\psi}}^2 \\
&\qquad \qquad +\big(p-1\big)^2 \left[q - 1 - \frac{(n-p)}{(p-1)(n-1)}\right]\bigg\langle\nabla\abs{\nabla\psi},\frac{\nabla\psi}{\abs{\nabla\psi}}\bigg\rangle^2\Bigg]e^{-\frac{n-p}{(n-2)(p-1)}\psi}d\mu.
\end{split}
\end{equation}
{In particular, if $(p, q) \in \Lambda$}, we have
\begin{equation}
\label{idintf2}
e^{-\frac{n-p}{(n-2)(p-1)}S}\big(\Psi_q^p\big)'(S)- e^{-\frac{n-p}{(n-2)(p-1)}s}\big(\Psi_q^p\big)'(s)\geq 0.
\end{equation}
\end{theorem}
\begin{proof}
We integrate both sides of \eqref{bochgenf} on $U=\{s\leq\psi\leq S\}$ with respect to the measure $e^{-\frac{n-p}{(n-2)(p-1)}\psi}d\mu$. The reason behind the choice the weight will be clear in a moment. Integration by parts of the last term of \eqref{bochgenf} combined with $p$-harmonicity of $\psi$ and the identity $\nabla^2 \psi (\nabla \psi)=\abs{\nabla \psi} \nabla \abs{\nabla \psi}$ yields
\begin{equation}
\label{parti}
\begin{split}
&\int_U\nabla^2\abs{\nabla\psi}(\nabla\psi, \nabla\psi)\abs{\nabla\psi}^{p-2}\abs{\nabla\psi}^{2-p}\abs{\nabla\psi}^{q(p-1)-4}e^{-\frac{n-p}{(n-2)(p-1)}\psi}d\mu= \\
&\qquad -\int_U\bigabs{\nabla\abs{\nabla\psi}}^2\abs{\nabla\psi}^{q(p-1)-3}e^{-\frac{n-p}{(n-2)(p-1)}\psi}d\mu \\
&\qquad-\big(q(p-1)-p-2\big)\int_U\langle\nabla\abs{\nabla\psi}, \nabla\psi\rangle^2\abs{\nabla\psi}^{q(p-1)-5}e^{-\frac{n-p}{(n-2)(p-1)}\psi}d\mu \\
&\qquad+\frac{n-p}{(n-2)(p-1)}\int_U\abs{\nabla\psi}^{q(p-1)-2}\langle\nabla\abs{\nabla\psi}, \nabla\psi\rangle e^{-\frac{n-p}{(n-2)(p-1)}\psi}d\mu \\
&\qquad+\int_{\partial U}\langle\nabla\abs{\nabla\psi}, \nu\rangle\abs{\nabla\psi}^{q(p-1)-2}e^{-\frac{n-p}{(n-2)(p-1)}\psi}d\sigma,
\end{split}
\end{equation}
where $\nu$ is the exterior unit normal to $U$. 
Noticing that 
\[
X=\big((p-1)(q-1)\big)\abs{\nabla\psi}^{q(p-1)-2}\nabla\abs{\nabla\psi},
\]
by a straightforward rearrangement of terms we obtain
\[
\begin{split}
&\int_U \Big(\dive(X)- \frac{n-p}{(n-2)(p-1)}\langle X, \nabla\psi\rangle\Big) e^{-\frac{n-p}{(n-2)(p-1)}\psi} d\mu= \\
&\bigintsss_U\big((p-1)(q-1)\big)\abs{\nabla\psi}^{q(p-1)-3}\Bigg(\abs{\nabla^2\psi}^2 + \big(q(p-1) - 3\big)\bigabs{\nabla\abs{\nabla\psi}^2} \\
&\quad +(p-2)\big(q(p-1)-p-1\big)\bigg(\frac{\langle\nabla\abs{\nabla\psi},\nabla\psi\rangle}{\abs{\nabla\psi}}\bigg)^2\Bigg)e^{-\frac{n-p}{(n-2)(p-1)}\psi}d\mu \\
&\!\!-(p-2)\int_{\partial U} \langle X, \nu\rangle e^{-\frac{n-p}{(n-2)(p-1)}\psi} d\sigma.
\end{split}
\]
Finally, consider the vector field
\[
Y=X e^{-\frac{n-p}{(n-2)(p-1)}\psi}.
\]
Since
\[
\dive Y=\Big(\dive X- \frac{(n-p)}{(n-2)(p-1)}\langle X, \nabla\psi\rangle\Big)e^{-\frac{n-p}{(n-2)(p-1)}\psi},
\]
Divergence Theorem gives
\begin{align*}
\int_U\Big(\dive(X)- \frac{n-p}{(n-2)(p-1)}\langle X, \nabla\psi\rangle\Big)e^{-\frac{n-p}{(n-2)(p-1)}\psi} d\mu &\!= \int_{\partial U} \langle Y, \nu\rangle d\sigma\\
&\!=\!\!\!\int_{\partial U} \langle X, \nu\rangle e^{-\frac{n-p}{(n-2)(p-1)}\psi} d\sigma.
\end{align*}
By Proposition \ref{deri} we have
\begin{equation*}
\int_{\partial U}\langle X, \nu\rangle e^{-\frac{n-p}{(n-2)(p-1)}\psi} d\sigma=\Big(e^{-\frac{n-p}{(n-2)(p-1)}S}\big(\Psi_q^p\big)'(S)- e^{-\frac{n-p}{(n-2)(p-1)}s}\big(\Psi_q^p\big)'(s)\Big),
\end{equation*}
and thus, by rearranging terms and dividing for $(p-1)$, we obtain
\begin{equation}
\begin{split}
&\Big(e^{-\frac{n-p}{(n-2)(p-1)}S}\big(\Psi_q^p\big)'(S)- e^{-\frac{n-p}{(n-2)(p-1)}s}\big(\Psi_q^p\big)'(s)\Big)= \\
&{\big(q-1\big)}\bigintsss_{\{{s\leq\psi\leq S}\}}\abs{\nabla\psi}^{q(p-1)-3}\Bigg[\abs{\nabla^2\psi}^2 + \Big(q(p-1)-3\Big)\bigabs{\nabla\abs{\nabla\psi}}^2 \\
&+(p-2)\Big(q(p-1)-p-1\Big)\left\langle\nabla\abs{\nabla\psi}, \frac{\nabla\psi}{\abs{\nabla \psi}}\right\rangle^2\Bigg]e^{-\frac{n-p}{(n-2)(p-1)}\psi}d\mu.
\end{split}
\end{equation}
Applying \eqref{ineqf} to the integrand on the right hand side of the above identity, with $\psi = f$ and $x\in U$, we obtain, after some elementary algebra, that 
\begin{equation}
\label{squarebra}
\begin{split}
&\!\!\!\!\!\abs{\nabla^2\psi}^2 + \Big((p-1)q-3\Big)\bigabs{\nabla\abs{\nabla\psi}}^2 
+(p-2)\Big(q(p-1)-p-1\Big)\bigg\langle\nabla\abs{\nabla\psi},\frac{\nabla\psi}{\abs{\nabla\psi}}\bigg\rangle^2= \\
=\,\,\,\, &\bigabs{\nabla^2_T f - \frac{\Delta_T f}{n-1}\,g_T}^2 + \Big(q(p-1)-1\Big)\bigabs{\nabla_T\abs{\nabla\psi}}^2  \\
\,\,+\, &\big(p-1\big)^2 \left[q - 1 - \frac{(n-p)}{(p-1)(n-1)}\right]\bigg\langle\nabla\abs{\nabla\psi},\frac{\nabla\psi}{\abs{\nabla\psi}}\bigg\rangle^2,
\end{split}
\end{equation}
where the tangential elements are referred to the level set $\{ \psi = \psi(x)\}$.
The proof is thus completed.

%INIZIO COMMENTO
\begin{comment}
We now turn to prove that the right hand side of \eqref{derif} is greater or equal than zero for $(p,q) \in \Lambda$. By classic Kato inequality
\[
\abs{\nabla^2\psi}\geq \bigabs{\nabla\abs{\nabla\psi}}^2,
\]
we immediately get that 
\begin{equation}
\label{kato1}
\abs{\nabla^2\psi}^2 + \Big((p-1)q - 3\Big)\bigabs{\nabla\abs{\nabla\psi}}^2\geq \Big((p-1)q-2\Big)\bigabs{\nabla\abs{\nabla\psi}}^2.
\end{equation}
We then divide the cases $p\in (1, 2]$ and $p\in(2, n)$. When $p\in (1, 2]$, if  $2/(p-1)\leq q\leq(p+1)/(p-1)$, the integrand on the right hand side of \eqref{idintf1} is nonnegative. Otherwise, if $q\geq (p+1)/(p-1)$, we have by Cauchy-Schwartz 
\begin{equation}
\label{cs1}
(p-2)\Big((p-1)q - p - 1\Big)\bigg(\frac{\nabla\abs{\nabla\psi}, \nabla\psi}{\abs{\nabla\psi}}\bigg)^2\geq (p-2)\Big((p-1)q - p - 1\Big)\bigabs{\nabla\abs{\nabla\psi}}^2, 
\end{equation}
and thus by Kato inequality we get that the right hand side of the \eqref{idintf1} is nonnegative if 
\[
((p-1)q-2) + (p-2)((p-1)q-p-1)\geq 0,
\]
that is, if $(p-1)q\geq p$. We thus obtain that when $p\in (1, 2]$ the right hand side of \eqref{idintf1} is nonnegative if $q\geq 2/(p-1)$.

Let now $p\in (2, n)$. If $(p-1)q\geq p+1$ we are done. Otherwise, as before, by Cauchy-Schwartz inequality we get that the right hand side of \eqref{idintf1} is nonnegative if $(p-1)q\geq p$. Thus, we get that that if  $p\in (2, n)$ and $q\geq p/(p-1)$ the right hand side of \eqref{idintf1} is nonnegative.
\end{comment}

%FINE COMMENTO

\end{proof}
We have now all the ingredients to conclude the proof of Theorem \ref{mainconf} .
%TEOREMA COMMENTATO
\begin{comment}
\begin{theorem}
\label{derpsif}
Let $(p, q) \in \Lambda$. Then the following facts hold true.
\begin{enumerate}
\item[(i)]
For any $s \geq 0$ the derivative of $\Psi_q^p$ is computed as follows.
\begin{equation}
\label{derpsi2}
\begin{split}
\big(\Psi_q^p\big)'(s) 
=-{\big(q-1\big)}&e^{\frac{n-p}{(n-2)(p-1)}s}\bigintsss_{\{{\psi\geq s}\}}\abs{\nabla\psi}^{(p-1)q-3}\Bigg[\left\vert\nabla^2_T f - \frac{\Delta_T f}{n-1}\,g_T\right\vert^2 \\
&\!\!\!\!\!\!\!\!+ \Big((p-1)q-1\Big)\bigabs{\nabla_T\abs{\nabla\psi}}^2 \\
&\!\!\!\!\!\!\!\! +\big(p-1\big) \left((p-1)q - p + \frac{p-1}{n-1} \right)\bigg\langle\nabla\abs{\nabla\psi},\frac{\nabla\psi}{\abs{\nabla\psi}}\bigg\rangle^2\Bigg] d\mu 
\leq 0.
\end{split}
\end{equation}
\item[(ii)]
If $\big(\Psi_p^q\big)'(s_0)=0$ for some $(p, q)\in\Lambda$ and $s_0\geq 0$, the manifold $(\{\psi\geq s_0\}, g)$ is isometric to $\big([s_0, + \infty) \times \{\psi=s_0\}, d\rho\otimes d\rho + g_{|\{\psi=s_0\}}\big)$, where $\rho$ is the $g$-distance to $\{\psi=s_0\}$ and $\psi$ is an affine function of $\rho$. Moreover $\left(\{\psi =s_0\}, g_{\mid\{\psi = s_0\}}\right)$ is a constant curvature sphere.
\end{enumerate}
\end{theorem} 
\end{comment}

%FINE TEOREMA COMMENTATO

\begin{proof}[Conclusion of the proof of Theorem \ref{mainconf}]
We first prove that $\big(\Psi_q^p\big)'(s)\leq 0$ for any $s\geq 0$, adapting an argument used in the proof of \cite[Theorem 1.1]{ColdMin}. Indeed, by \eqref{idintf2}, for any $S\geq s$ we have 
\[
(\Psi_q^p)'(S)\geq  e^{\frac{n-p}{(n-2)(p-1)}(S-s)}(\Psi_q^p)'(s).
\]
Integrating the above identity with respect to $S$, we get
\[
\Psi_q^p(S)\geq{\frac{(n-2)(p-1)}{(n-p)}} \left(e^{\frac{n-p}{(n-2)(p-1)}(S-s)}-1\right){\Psi_q^p}'(s) + \Psi_q^p (s)
\]
for any $S>s$. If, for some $s>0$, $(\Psi_q^p)'(s)$ was strictly positive, then, letting $S\to\infty$, we would obtain $\Psi_q^p\to\infty$, against the uniform boundedness of $\Psi_q^p$ proved in Lemma \ref{bound}.

Finally, notice that since $\Psi_p^q$ is a continuous, bounded, nonincreasing function, we have that $\big(\Psi_p^q)'(S)\to 0$ as $S\to+\infty$. Formula \eqref{derpsi2} is thus proved by passing to the limit as $S\to +\infty$ in \eqref{idintf1}.  
\newline{}

Assume now $\left(\Psi_q^p\right)'(s_0)=0$ for some $(p, q) \in \Lambda$. If $q > 1 + (n-p)/[(p-1)(n-1)]$, then $\nabla \abs{\nabla \psi} = 0$. Thus, plugging this information in the Bochner-type formula given by Lemma \ref{bochp}, we get $\nabla^2\psi = 0$. The isometry with the Riemannian product then follows from \cite[Theorem 4.1-(i)]{AgoMaz1}. 

If $q = 1 + (n-p)/[(p-1)(n-1)]$, then we just have
\[
\bigabs{\nabla^2_T f - \frac{\Delta_T f}{n-1}\,g_T}^2  = 0 
\]
and
\[
\bigabs{\nabla_T\abs{\nabla f}}^2 =0.
\]
But then, by (ii) in Proposition \ref{ineq},
the metric $g$ on $\{\psi \geq s_0\}$ has a warped product structure
\[
g = d\rho \otimes d\rho + \eta^2 (\rho) g_{\mid\{\psi= s_0\}},
\]
for some positive warping function $\eta$. Moreover, by expression \eqref{eta}, $\psi$ and $\rho$ share the same level sets. In particular, the second fundamental form of the level sets of $\psi$ satisfies
\[
h_{ij}=\frac{1}{2}\frac{\partial g_{ij}}{\partial \rho} = \frac{d \log \eta}{d \rho} g_{ij},
\] 
and thus, taking the trace, we see that the level sets of $\psi$ have constant mean curvature. Employing now expression \eqref{derif} for the derivative of $\Psi_q^p$, we deduce that the mean curvature of $\{\psi =s\}$ is constantly zero for any $s \geq s_0$, observing also that by expression \eqref{derpsi2} the derivative of $\Psi_q^p$ vanishes for any $s \geq s_0$. By \eqref{meancurv2}, we then obtain $\langle \nabla \abs{\nabla \psi}, \nabla \psi\rangle =0$. In particular, we have $\nabla\abs{\nabla \psi}=0$ and the isometry follows as in the case $q > 1 + (n-p)/[(p-1)(n-1)]$. 

The isometry of $\{\psi = s_0\}$ with a constant curvature sphere follows by \cite[Theorem 4.1-(ii)]{AgoMaz1}, once noticed that the second equation in \eqref{prob_ex_rif} is equivalent to the one considered in the aforementioned paper  as $\nabla^2 \psi = 0$, that is indeed our case.
\end{proof}

\section{Proof of Theorem \ref{sharpbound}}
Our second main result is a further consequence of the Bochner-type differential identity \eqref{pBoch2}, a strong maximum principle for $\abs{\nabla\psi}$. 
The key point in our proof of Theorem \ref{sharpbound} is the fact that identity \eqref{pBoch2} implies $\abs{\nabla\psi}^p$ to be sub-solution of an elliptic equation. Indeed, let us define the operator $\mathcal{L}: C^2(M) \to \R$ as
\begin{equation}
\label{l}
\mathcal{L}f=\Delta f + (p-2)\nabla^2 f\bigg( \frac{\nabla\psi}{\abs{\nabla\psi}}, \frac{\nabla \psi}{\abs{\nabla \psi}}\bigg) - \frac{n-p}{n-2}\langle \nabla f, \nabla\psi\rangle,
\end{equation}
where $\psi$ is a solution of \eqref{prob_ex_rif}. Then the following lemma holds. 
\begin{lemma}
\label{elliptic2}
Let $(M, g, \psi)$ be a solution to \eqref{prob_ex_rif}, and let $\mathcal{L}$ be the differential operator defined in \eqref{l}. Then the following facts hold true.
\begin{enumerate}
\item[i)] The operator $\mathcal{L}$ is elliptic non-degenerate.
\item[ii)] 
\begin{equation}
\label{sub}
\mathcal{L}(\abs{\nabla \psi}^p)\geq 0.
\end{equation}
\item[iii)] 
\begin{equation}
\label{barrier}
\mathcal{L} \left(e^{\frac{n-p}{(n-2)(p-1)}\psi}\right) = 0.
\end{equation}
\end{enumerate}
\end{lemma}
\begin{proof}
Let $x\in M$, and choose at $x$ an orthonormal frame $\{e_1, \dots,\, e_{n-1},\, e_n=\nabla\psi/ \abs{\nabla\psi}\}$, recalling that $|\nabla\psi|\neq 0$ on $M$. Then, the higher order term in $\mathcal{L}$ can be computed at $x$ as 
\[
\mathbb{L}(f)=\nabla_{e_1}\big(\nabla_{e_1} f\big) + \dots + \nabla_{e_{n-1}}\big(\nabla_{e_{n-1}} f\big) + (p-1)\nabla_{e_n}\big(\nabla_{e_n} f\big). 
\]
Since $p>1$, the operator $\mathcal{L}$ is then clearly elliptic non-degenerate.
\newline{}

Let us check that $\abs{\nabla\psi}^p$ satisfies $\mathcal{L}(\abs{\nabla\psi}^p)\geq 0$. By \eqref{pBoch2},
we have
\[
\mathcal{L}(\abs{\nabla \psi}^p)=p\abs{\nabla\psi}^{p-2}\left(\abs{\nabla^2 \psi} + p(p-2)\Big\langle \nabla \abs{\nabla\psi}, \frac{\nabla \psi}{\abs{\nabla \psi}}\Big\rangle^2\right).
\]
By the standard Kato inequality $\abs{\nabla ^2 \psi}^2 \geq \bigabs{\nabla \abs{\nabla \psi}}^2$, we obtain 
\[
\mathcal{L}(\abs{\nabla\psi}^p)\geq p\abs{\nabla\psi}^{p-2}\left(\abs{\nabla\abs{\nabla\psi}}^2 + p(p-2)\Big\langle \nabla \abs{\nabla\psi}, \frac{\nabla \psi}{\abs{\nabla \psi}}\Big\rangle^2\right),
\]
and, since 
\[
\abs{\nabla\abs{\nabla\psi}}^2 + p(p-2)\Big\langle \nabla \abs{\nabla\psi}, \frac{\nabla \psi}{\abs{\nabla \psi}}\Big\rangle^2 = \bigabs{\nabla_T \abs{\nabla \psi}}^2 + (p-1)^2 \Big\langle \nabla \abs{\nabla\psi}, \frac{\nabla \psi}{\abs{\nabla \psi}}\Big\rangle^2 \geq 0,
\]
we get \eqref{sub}. As before, $\nabla_T$ denotes the tangential gradient with respect to the level set $\{\psi = \psi(x)\}$, $x$ being the generic point where computations are carried out.
\newline{}

Let us now compute $\mathcal{L} (e^{\frac{n-p}{(n-2)(p-1)}\psi})$.
First, we have
\begin{equation}
\label{delta}
\Delta (e^{\frac{n-p}{(n-2)(p-1)}\psi})=\frac{n-p}{(n-2)(p-1)}e^{\frac{n-p}{(n-2)(p-1)}\psi}\bigg(\frac{n-p}{(n-2)(p-1)}\abs{\nabla\psi}^2 + \Delta\psi\bigg)
\end{equation}
Using 
\begin{equation}
\label{plappsi}
\Delta_p\psi=\abs{\nabla\psi}^{p-2}\bigg(\Delta \psi + (p-2)\nabla^2 \psi \left(\frac{\nabla\psi}{\abs{\nabla\psi}}, \frac{\nabla\psi}{\abs{\nabla\psi}}\right)\bigg)=0,
\end{equation}
where the last equality follows from the $p$-harmonicity of $\psi$, we obtain, from \eqref{delta} and the fact that $\abs{\nabla\psi}\neq 0$ 
\begin{align}
\label{delta1}
&\Delta (e^{\frac{n-p}{(n-2)(p-1)}\psi})=\\
\nonumber
&=\frac{n-p}{(n-2)(p-1)}e^{\frac{n-p}{(n-2)(p-1)}\psi}\left(\frac{n-p}{(n-2)(p-1)}\abs{\nabla\psi}^2 - (p-2) \nabla^2 \psi \left(\frac{\nabla\psi}{\abs{\nabla\psi}}, \frac{\nabla\psi}{\abs{\nabla\psi}}\right)\right).
\end{align}
The second term to be computed in $\mathcal{L}(e^{\frac{n-p}{(n-2)(p-1)}\psi})$ is 
\begin{equation}
\label{hess}
\begin{split}
\nabla^2\big(e^{\frac{n-p}{(n-2)(p-1)}\psi}\big)\bigg(\frac{\nabla\psi}{\abs{\nabla\psi}}, \frac{\nabla\psi}{\abs{\nabla\psi}}\bigg) = \frac{n-p}{(n-2)(p-1)}e^{\frac{n-p}{(n-2)(p-1)}\psi}\bigg[&\frac{n-p}{(n-2)(p-1)}\abs{\nabla\psi}^2 
\\
 + &\nabla^2 \psi \bigg(\frac{\nabla\psi}{\abs{\nabla\psi}}, \frac{\nabla\psi}{\abs{\nabla\psi}}\bigg)\bigg],
\end{split}
\end{equation}
and the last one is 
\begin{equation}
\label{scal}
\langle\nabla e^{\frac{n-p}{(n-2)(p-1)}\psi}, \nabla\psi\rangle=\frac{n-p}{(n-2)(p-1)}e^{\frac{n-p}{(n-2)(p-1)}\psi}\abs{\nabla\psi}^2.
\end{equation}
By \eqref{delta1}, \eqref{hess} and \eqref{scal} we get $\mathcal{L}(e^{\frac{n-p}{(n-2)(p-1)}\psi})=0$, as claimed.
\end{proof}

We are now in position to prove Theorem \ref{sharpbound}.

\begin{proof}[Proof of Theorem \ref{sharpbound}]
We claim that
\begin{equation}
\label{sharpboundf}
\abs{\nabla \psi}(x) \leq \sup_{\{\psi=s\}}\abs{\nabla \psi}
\end{equation}
for any $x \in \{\psi \geq s\}$. This clearly suffices to prove the monotonicity of $V_\infty^p$. Indeed for any $s_1\leq s_2$, \eqref{sharpboundf} implies
\begin{align}
\sup_{\{\psi=s_1\}}\abs{\nabla \psi}\geq \sup_{\{\psi\geq s_1\}} \abs{\nabla \psi}\geq \sup_{\{\psi\geq s_2\}} \abs{\nabla \psi}\geq \sup_{\{\psi=s_2\}} \abs{\nabla \psi}.
\end{align}
Recall that by Proposition \ref{mono_and} there exists $C>0$ such that
\[
\abs{\nabla\psi}\leq C
\] 
on $M$.
Consider then, for such a constant $C$ and for $S > s$, the function
\[
w=\abs{\nabla\psi}^p - \sup_{\{\psi = s\}}\abs{\nabla\psi}^p - C^p{e^{\frac{n-p}{(n-2)(p-1)}(\psi - S)}}
\]
defined on $\{s \leq \psi\leq S\}$. Clearly, we have
\[
\sup_{\{\psi = s\} \cup \{\psi = S\}}w \leq 0,
\] 
and thus, since by Lemma \ref{elliptic2} $\mathcal{L}$ is a uniformly elliptic operator on $\{s \leq \psi\leq S\}$ and $\mathcal{L}(w)\geq 0$ on the same bounded domain, the Maximum Principle applies and yields
\[
\abs{\nabla\psi}^p\leq \sup_{\{\psi = s\}}\abs{\nabla\psi}^p + C^p{e^{\frac{n-p}{(n-2)(p-1)}(\psi - S)}}
\]
on $\{s \leq \psi \leq S\}$. Computing the above inequality on a fixed $x\in\{s \leq \psi\leq S\}$, and passing to the limit as $S\to + \infty$ yields the claim.

Suppose that for some $s_1 > s$ there holds $\Psi_\infty^p(s_1)=\Psi_\infty^p(s)$ and let $x_{s_1}$ the maximum point of $\abs{\nabla \psi}$ on $\{\psi = s_1\}$. In particular, we have 
\[
\abs{\nabla \psi}^p(x_{s_1})=\sup_{\{\psi =s\}}\abs{\nabla \psi}^p.
\]
Let then $S>s_1$. We have
\[
\abs{\nabla\psi}^p(x_{s_1})=\sup_{\{\psi = s\} \cup \{\psi=S\}}\abs{\nabla\psi}^p,
\]
where we have used the just proved inequality \eqref{sharpboundf}. Thus, since by Lemma \ref{elliptic2}\, \mbox{$\mathcal{L}(\abs{\nabla\psi})^p \geq 0$}, we obtain by Strong Maximum Principle that $\abs{\nabla \psi}^p$ is constant on $\{s \leq \psi \leq S\}$. Since $S$ was arbitrarily big, we actually get that $\abs{\nabla \psi}$ is constant on the whole $\{ \psi \geq s\}$ and therefore $\bigabs{\nabla \abs{\nabla\psi}}=0$. By the Bochner-type identity \eqref{pBoch2} we then get $\abs{\nabla^2 \psi} = 0$, and the rigidity follows by \cite[Theorem 4.1-(i)]{AgoMaz1}.
\newline{}

We now turn to prove the second part of Theorem \ref{sharpbound}.
Let $x_s$ as in the statement. We have just proved that
\[
\abs{\nabla \psi}^p (x_s)\geq \abs{\nabla \psi}^p (x)
\]
for any $x \in \{\psi \geq s\}$. This implies that the derivative with respect to the normal vector $\nu = \nabla \psi /\abs{\nabla \psi}$ satisfies
\begin{equation}
\label{dernormpsi}
\frac{\partial}{\partial \nu} \abs{\nabla \psi}^p(x_s) \leq 0.
\end{equation}
Using \eqref{meancurv2} we get
\[
\left\langle \nabla \abs{\nabla \psi}^p, \frac{\nabla \psi}{\abs{\nabla \psi}} \right\rangle (x_s) = - \frac{p}{p-1}\left(\abs{\nabla \psi}^p H\right)(x_s),
\]
where $H$ is the mean curvature of the set $\{\psi=x_s\}$ and inequality \eqref{hgineq} follows.

Assume now that such an inequality holds with  equality sign.
Then, in particular, the normal derivative in \eqref{dernormpsi} vanishes. But since $x_s$ is a global maximum value for $\abs{\nabla\psi}^p$ on $\{s \leq \psi \leq S\}$ for any $S>s$, and $\abs{\nabla \psi}^p$ is subsolution of the elliptic equation $\mathcal{L} {f} = 0$ by \eqref{sub}, Hopf's lemma (see e.g. \cite[Lemma 3.4]{gilbarg}) implies that $\abs{\nabla \psi}^p$ is constant on this region. By the arbitrariness of $S$ we infer that $\abs{\nabla \psi}$ is constant on the whole $\{\psi \geq s\}$ and finally the rigidity part follows as before. 
\end{proof}

\textbf{Acknowledgements.}
The authors would like to thank J. Xiao for his interest in our work and for stimulating discussions during the preparation of the manuscript, and X. Zhong for pointing out some flaws. The authors are members of the Gruppo Nazionale per l’Analisi Matematica, la Probabilità e le loro Applicazioni (GNAMPA) of the Istituto Nazionale di Alta Matematica (INdAM) and are partially founded by the GNAMPA Project “Problemi sovradeterminati e questioni di ottimizzazione di forma". The paper was partially completed during the authors’ attendance to the program “Geometry and relativity” organized by the Erwin Schroedinger International Institute for Mathematics and Physics (ESI).

\end{document}